\documentclass[aop,seceqn,citesort,MSNbibl,dvips]{arximspdf}
\usepackage{mathbh}

\doi{10.1214/10-AOP622}
\volume{40}
\issue{1}
\pubyear{2012}
\firstpage{339}
\lastpage{371}

\makeatletter

\newcommand{\bigtimes}{\mathop{\,\mbox{\parbox[c][9pt][b]{18pt}{\fontsize{18}{18}\selectfont{$\times$}}}\!\!}}

\newtheorem{theorem}{Theorem}[section]
\newtheorem{lemma}[theorem]{Lemma}
\newtheorem{proposition}[theorem]{Proposition}

\newproclaim{remark}{Remark}

\newcommand{\bigforall}{\mathop{\hspace*{3pt}\rule[3pt]{5pt}{0.5pt}\hspace*{-10pt}\bigvee}_{i=1}^M}
\newcommand{\bigforalll}{\mathop{\hspace*{3pt}\rule[3pt]{4pt}{0.5pt}\hspace*{-8pt}\bigvee}_{i=1}^M}

\newcommand{\kk}{\kappa}
\newcommand{\G}{\Gamma}
\newcommand{\h}{\eta}
\newcommand{\R}{{\mathbb R}} 
\newcommand{\N}{{\mathbb N}} 
\newcommand{\Z}{{\mathbb Z}} 
\newcommand{\E}{{\mathbb E}} 

\newcommand{\un}[1]{\underline #1}
\newcommand{\BB}{{\mathcal B}}
\newcommand{\DD}{{\mathcal D}}
\newcommand{\EE}{{\mathcal E}}
\newcommand{\FF}{{\mathcal F}}
\newcommand{\GG}{{\mathcal G}}
\newcommand{\MM}{{\mathcal M}}
\newcommand{\NN}{{\mathcal N}}
\newcommand{\OO}{{\mathcal O}}
\newcommand{\QQ}{{\mathcal Q}}

\newcommand{\s}{{\sigma}}
\newcommand{\m}{{\mu}}

\newcommand{\cA}{{\mathcal A}}
\newcommand{\cB}{{\mathcal B}}
\newcommand{\cD}{{\mathcal D}}
\newcommand{\cF}{{\mathcal F}}
\newcommand{\cM}{{\mathcal M}}
\newcommand{\cN}{{\mathcal N}}
\newcommand{\cQ}{{\mathcal Q}}
\newcommand{\cS}{{\mathcal S}}

\newcommand{\capa}{\operatorname{cap}}
\newcommand{\wt}{\widetilde}

\newcommand{\bmm}{{\underline m}}
\newcommand{\bA}{{\mathbf A}}

\newcommand{\df}{\stackrel{\Delta}{=}}

\newcommand{\so}{o(1)}
\newcommand{\vm}{\bmm}
\newcommand{\ui}{\underline{i}}

\makeatother

\begin{document}
\begin{frontmatter}

\title{Pointwise estimates and exponential laws in metastable
systems via coupling methods\thanksref{T1,T2}}
\runtitle{Exponential laws via coupling}

\thankstext{T1}{Supported in part through a~grant by
the German--Israeli Foundation (GIF).}
\thankstext{T2}{Supported in part by the DFG in the SFB 611 and
the Hausdorff Center for Mathematics.}

\begin{aug}
\author[A]{\fnms{Alessandra} \snm{Bianchi}\ead[label=e1]{alessandra.bianchi7@unibo.it}},
\author[B]{\fnms{Anton} \snm{Bovier}\ead[label=e2]{bovier@uni-bonn.de}}
and
\author[C]{\fnms{Dmitry} \snm{Ioffe}\corref{}\ead[label=e3]{ieioffe@technion.ac.il}}
\runauthor{A. Bianchi, A. Bovier and D. Ioffe}
\affiliation{Weierstrass-Institute for Applied Analysis and
Stochastics, Institut f\"{u}r Angewandte Mathematik, Rheinische
Friedrich--Wilhelms-Universit\"{a}t and Faculty of Industrial
Engineering and
Management, Technion}
\address[A]{A. Bianchi\\
Weierstrass-Institute for Applied Analysis\\
\quad and Stochastics\\
Mohrenstrasse 39\\
10117 Berlin\\
Germany\\
\printead{e1}}
\address[B]{A. Bovier\\
Institut f\"{u}r Angewandte Mathematik\\
Rheinische Friedrich--Wilhelms-Universit\"{a}t\\
Endenicher Allee 60\\
53115 Bonn\\
Germany\\
\printead{e2}}
\address[C]{D. Ioffe\\
Faculty of Industrial Engineering\\
\quad and Management\\
Technion---Israel Institute of Technology\\
Haifa 32000\\
Israel\\
\printead{e3}}
\end{aug}

\received{\smonth{3} \syear{2010}}
\revised{\smonth{9} \syear{2010}}

%
\begin{abstract}
We show how coupling techniques can be used in some metastable systems to
prove that mean metastable exit times are almost constant as functions
of the
starting microscopic configuration within a~``meta-stable set.''
In the example of the
Random Field Curie Weiss model, we show that these ideas can also be used
to prove asymptotic exponentiallity of normalized metastable escape times.
\end{abstract}

%

\end{frontmatter}

\section{Introduction}

\subsection{The problem}\label{subsection.1.1}
Metastable systems
are characterized by the fact that the state space can be decomposed
into several disjoint subsets, with the property that transition times
between subspaces are long compared to characteristic mixing times
within each subspace.
The mathematically rigorous analysis of
Markov processes exhibiting metastable
behavior
was first developed
in the \textit{large deviation theory} of Freidlin and Wentzell \cite
{FW,OV}. This
approach yields logarithmic asymptotics of transition times and other
quantities of interest. Over the last decade, a~\textit{potential theoretic
approach}~\cite{BEGK02,Bo5} to metastability was developed that in many
instances yields more precise asymptotics, and in particular the exact
prefactors of exponential terms.

In this work we study metastability for a~class of
stochastic Ising models. The main objective is to extend the potential
theoretical
approach for deriving asymptotics of transition times\vadjust{\goodbreak}
for processes starting from individual microscopic configurations, and,
subsequently, for studying exponential scaling laws for these
transition times.

So far the existing methods work well in the following situations:
\begin{longlist}[(3)]
\item[(1)] The process is strongly recurrent in the sense that it
visits an
individual atom of the state space in each metastable state many times
with overwhelming probability \textit{before} a~metastable transition happens.
This situation occurs, for example, in Markov chains with finite state
space, and on
discrete state space, such as $\Z^d$, in the presence of a~confining potential.
\item[(2)] In models where strong symmetries allow the analysis of the dynamics
through a~\textit{lumped chain} that satisfies the requirements of (1). This
situation occurs, for example, in mean field models such as the
Curie--Weiss model~\cite{CGOV}
and the Curie--Weiss model with random magnetic fields that take only finitely
many values~\cite{MP,BEGK01}.
\item[(3)] In situations where the process returns often to small neighborhoods,
$\OO_{\varepsilon}(x)$ of points, $x$, in a~metastable state where
the oscillations of harmonic
functions on these neighborhoods can be made arbitrarily small. This is the
case in finite and some infinite-dimensional diffusion processes
\cite{MOS89,BEGK04}.
\end{longlist}

One would expect that   situation (3) also arises in a~wide variety of
stochastic Ising models or stochastic particle systems exhibiting metastable
behavior. \textit{Proving} the respective regularity properties of
microscopic harmonic
functions appears, however, to be a~difficult issue in general.

The purpose of
the present paper is to develop an approach to this
problem via coupling techniques that allow to cover at least some
interesting situations.

A key idea of the potential theoretic approach is to express quantities
of physical interest in terms of \textit{capacities} and to use
\textit{variational principles} to compute the latter. A~fundamental
identity used systematically in this approach is a~representation
formula for the Green's function, $g_B(x,y)$,
with Dirichlet conditions in a~set~$B$, that reads (in the context of
arbitrary discrete state space)
%
%
\begin{equation}\label{intro.1}
g_B(x,y) = \mu(y)\frac{h_{x,B}(y)}{\capa(x,B)},
\end{equation}
where $B$ is a~subset of the configuration space,
$h_{x,B}(y) = h_{\{ x\}, B}$ and $h_{A, B}$ is the equilibrium
potential, that is,
%
%
\begin{equation}
\label{hAB}
h_{A,B}(y) =
\cases{
1, &\quad if $y\in A$,\cr
0, &\quad if $y\in B$,\cr
{\mathbb P}_y(\tau_A <\tau_B), &\quad otherwise.}
\end{equation}
We use
\[
\tau_C = \min\{ t >0\dvtx x (t)\in C\}
\]
for the first hitting times of sets $C$, and
$\capa(A,B)$ is the capacity between the sets $A$ and $B$; $\capa
(x,B) =\capa(\{ x\},B)$.\vadjust{\goodbreak}

Equation (\ref{intro.1}) immediately leads to a~formula for the
mean hitting time $\E_x\tau_B$ of $B$, for the process starting in
$x$. However, the resulting
expression for $\E_x \tau_B$ is
useful as long as the ratio appearing in (\ref{intro.1}) is under
control and is
not seriously
of the form $0/0$.

To be more precise,
it may happen that $h_{x,B}(y)= f(A) h_{A,B}(y)$ and\break $\capa(x, B)=
f(A) \capa(A,B)$, for ``macroscopic''
sets $A\ni x$.
Then
%
%
\begin{equation}
\label{eq:symmetry}
\frac{h_{x,B}(y)}{\capa(x,B)}=\frac{h_{A,B}(y)}{\capa(A,B)},
\end{equation}
but except
in cases where (\ref{eq:symmetry}) is manifest by some symmetry, it
will be very hard to
establish
such relations by a~direct pointwise estimation of numerator and
denominator in
(\ref{intro.1}).

Examples where this problem occurs are diffusion processes in $d>1$,
Glauber dynamics in the case of finite temperature, etc. In such
cases, a~useful version can be extracted by averaging
equation (\ref{intro.1}) with respect to $x$ after multiplying both sides
by $\capa(x,B)$ over as suitable neighborhood $A\equiv A_x$. This yields
the formula
%
%
\begin{equation}\label{intro.2}
\E_{\nu_A}\tau_B =\frac{1}{\capa(A,B)}\sum_{y} h_{A,B}(y)\m(y),
\end{equation}
where $\nu_A$ is a~specific probability distribution on $A$. Actually
(\ref{intro.2}) can be derived without a~recourse to
(\ref{intro.1}): if $P$ is the transition kernel of a~reversible
Markov chain $x(t)$, then
the equilibrium potential $h_{A,B}$ is harmonic outside $A\cup B$;
$(I-P)h_{A,B} = Lh_{A,B} = 0$.
Thus,
%
%
\begin{equation}
h_{A,B} (y) = \sum_{x\in A} g_B (y ,x)L h_{A,B} (x)
\end{equation}
for all $y\notin B$. By reversibility
$\mu(y) g_B (y,x ) = \mu(x) g_B (x, y)$, and
it follows that
%
%
\begin{equation}
\sum_{y \notin A\cup B} \mu(y)h_{A,B} (y) = \sum_{x\in A}\mu(x) L
h_{A,B} (x) \E_x \tau_B ,
\end{equation}
which is (\ref{intro.2}) with $\nu_A (x) = \mu(x) L h_{A,B} (x) /
\capa(A, B)$.

The point is that the
right-hand side of (\ref{intro.2}) can be evaluated in many cases of
interest when
formula (\ref{intro.1}) suffers from the problem discussed above.
This has been demonstrated recently
in two examples, the Glauber dynamics of the random field Curie--Weiss
model at finite temperature~\cite{BBI08}, and the Kawasaki dynamics in
the zero temperature limit on volumes that diverge exponentially with
the inverse temperature~\cite{BdHS08}.

An obvious question is whether the mean hitting time of $B$
really depends on the specific initial distribution $\nu_A$ or
whether, for all $z\in A$, $\E_z\tau_B$ is equal to $ \E_{\nu_A}
\tau_B$ up to a~small error.
This question, and related one concerning other functions of initial
conditions is of much further reaching importance. In particular, it is
relevant for proving the asymptotic exponentiallity of the transition time
using approximate renewal arguments. Let us mention that the same issue
also arises in the case of diffusion equations in the Wentzell--Freidlin
regime. Here, Martinelli and Scoppola~\cite{MS88},
Martinelli, Olivieri and Scoppola
\cite{MOS89} showed that solutions
of the stochastic differential equation starting at
two different points
in a~neighborhood of a~stable equilibrium and driven by the same
noise
are converging exponentially fast to each other with probability tending
to one. From this, they deduced regularity of exit probabilities
${\mathbb P}_x[\tau_B>t\E\tau_B]$ as functions of $x$ and hence
exponentiallity of~$\tau_B$
and asymptotic independence of $\E_x\tau_B$ of the starting point
$x\in A$.
Such a~strong contraction property is, however, not
available
in stochastic Ising models on the level of microscopic paths.

In the present paper, we will
develop a~method that allows us to obtain
similar results, at least in some cases,
with
an alternative and,
weaker input. It is based on coupling
techniques and allows us to turn the following simple
heuristic argument
into a~rigorous proof: the Markov chain should mix quickly
before it leaves a~substantial neighborhood of the starting point $x$;
since the mixing time is short compared to the hitting time $\tau_B$,
the mean of $\tau_B$ should be the same for all starting configuration
in $A$. Moreover, the chain will return many times to $A$ before
reaching $B$; by rapid mixing, the return times will be
essentially
i.i.d., hence the
number of returns will be geometric, and the scaled hitting time
will be exponential.

To demonstrate the usefulness of this approach, our key example will be the
Random Field Curie--Weiss model with
continuous distribution of the random fields. In that sense, the present
result is also a~completion of our previous paper~\cite{BBI08}.
Technically, the coupling construction we employ is based on
\cite{LLP08} and still contains model dependent elements. However, the
basic ideas are more general and will be of relevance for the treatment
of a~wider range of metastable systems.

The remainder of this paper is organized as follows. In the next
subsection, we describe a~general setting of Markov chains to which our
method applies. In Section~\ref{sec13}, we state our two main theorems.
In Section~\ref{Sect.2}, we recall the definition of Glauber dynamics
for the random field Curie--Weiss model and recall the main result from
\cite{BBI08}. In Section~\ref{S1}, we recall the coupling constructed
by Levin, Luczak and Peres for the standard Curie--Weiss model and show
how this can be modified to be useful in the random field model. We
then prove Theorem~\ref{pointwise.1}. In Section~\ref{Sect.4}, we show
how to prove the asymptotic exponentiallity of the transition times and
give the proof of Theorem~\ref{exponential.1}.

\subsection{Setting} \label{subsection.1.2}

In this subsection, we describe a~general setting in which our methods
can be applied.

In the sequel, $N$ will be a~large parameter. We consider (families of)
Markov processes, $\s(t)$, with finite state space,\vadjust{\goodbreak}
$\cS_N\equiv\{-1,1\}^N$, and transition probabilities $p_N$ that are
reversible w.r.t. a~(Gibbs) measure, $\mu_N$. Transition probabilities
$p_N$ always have the following structure: at each step, a~site
$x\in\Lambda$ is chosen with uniform probability $1/N$. Then the spin
at $x$ is set to $\pm1$ with probabilities $p^{\pm}_x (\s)$; $p^{+}_x
(\s) + p^{-}_x (\s)\equiv1$. In the sequel, we shall assume that there
exists $\alpha\in[1/2,1)$ such that
%
%
\begin{equation}
\label{eq:PUpdate}
\max_{x ,\s,\pm} p^{\pm}_x (\s) \leq\alpha.
\end{equation}
A key hypothesis is the existence of a~family of ``good'' mesoscopic
approximations of our processes. By this, we mean the following: there
is a~sequence of disjoint partitions, $\{\Lambda^{n}_1,\ldots,
\Lambda^{n}_{k_n}\}$, of $\Lambda\equiv\{1,\ldots,N\}$, and a~family of
maps, $\un m^{(n)}\dvtx\cS_N\rightarrow\G_n\subset\R^{n}$, given by
%
%
\begin{equation}\label{finer.1}
m_i^n(\s) =\frac1N\sum_{x\in\Lambda^n_i}\s_x .
\end{equation}
We will always think of these partitions as nested, that is,
$\{\Lambda^{n+1}_1,\ldots,\Lambda^{n+1}_{k_{n+1}}\}$ is a~refinement of
$\{\Lambda^{n}_1,\ldots, \Lambda^{n}_{k_n}\}$. On the other hand, to
lighten the notation, we will mostly drop the superscript and identify
$k_n=n$, and refer to the generic partition
$\Lambda_1,\ldots,\Lambda_n$. It will be convenient to introduce the
notation
\[
{\mathcal S}^n[\un m]\equiv(\un m^n)^{-1}(\un m) =\{ \sigma\dvtx
\un m^n (\sigma) = \un m\}
\]
for the set-valued inverse images of $\un m^n$. We think of the maps
$\un m^n$ as some block averages of our ``microscopic'' variables
$\s_i$ over blocks of decreasing (in $n$) ``mesoscopic'' sizes.

As is well known, the image process, $\un m^{n}(\s(t))$, is in general
not Markovian. However, there is a~canonical Markov process, $\un
m^{n}(t)$, with state space~$\G_n$ and reversible measure
$\QQ_n\equiv\mu_N\circ(\un m^n)^{-1}$, that is a~``best'' approximation
of $\un m^{n}(\s(t))$, in the sense that if $\un m^{n}(\s(t))$ is
Markov, then $\un m^{n}(t)=\un m^{n}(\s(t))$ (in law). For all $\un m,
\un m'\in\G_n$, the transition probabilities of this chain are given by
%
%
\begin{equation}\label{intro.5}
r_N(\un m,\un m')\equiv\frac1{\QQ_n(\un m)}
\mathop{\sum_{\s\in{\mathcal S}^n[\un m]}}_{\s'\in{\mathcal S}^n[\un
m']} \mu_N(\s)
p_N(\s,\s').
\end{equation}
In the models, we consider here the following two assumptions are
satisfied:
\begin{longlist}[(A.2)]
\item[(A.1)] The sequence of chains $\un m^{n}(t)$ approximates $\un
m^{n}(\s(t))$ in the strong sense that there exists
${\varepsilon}(n)\downarrow0$, as $n\uparrow\infty$, such that for any
$m,m'\in\G_n$,
%
%
\begin{equation}\label{intro.6}
\mathop{\max_{{\s\in S^n[\un m],\s'\in S^{n}[\un m']}}}_{r_N(\un m,
\un m')>0} \biggl| \frac
{p_N(\s,\s'){|{\mathcal S}^n[\un m']|}}{r_N(\un m,\un m')}-1
\biggr|\leq{\varepsilon}(n).
\end{equation}
\item[(A.2)] The microscopic flip rates satisfy: if $\un m^n (\s)
=\un m^n (\h)$
and $\s_x =\h_x$, then $p^\pm_x (\s) = p^\pm_x (\h)$.\vadjust{\goodbreak}
\end{longlist}
Note that our assumption (A.1) is
much stronger then the maybe more natural looking
\[
\max_{\s\in S^n[\un m]} \biggl| \frac
{\sum_{\s'\in S^n[\un m']}p(\s,\s')}{r_N(\un m,\un m')}-1\biggr|
\leq{\varepsilon}(n).
\]

Finally, we need to place us in a~``metastable'' situation.
Specifically, we will assume that there exist two disjoint sets
$A=\{\s\in\cS_N\dvtx\un m^{n_0}(\s)\in\bA\}$ and $
B=\{\s\in\cS_N\dvtx\un m^{n_0}(\s)\in{\mathbf B}\}$, for some $n_0$ and
sets $\bA,{\mathbf B}\subseteq \G_{n_0}$, a~constant $C>0$ and a~sequence $a_n<\infty$, such that, for all $n\geq n_0$ and for all
$\s,\h\in A$,
%
%
\begin{equation}\label{intro.10}
{\mathbb P}_{\s}\bigl[\tau_B<\tau_{\un m^n(\h)}\bigr]\leq
a_n e^{ -CN},
\end{equation}
where, with a~little abuse of notation, we denote by
$\tau_{\un m^n(\h)}$ the first hitting time of the set
$\cS^n[\un m^n(\h)]$.

\subsection{Main results}\label{sec13}

In the setting outlined above, we will prove the following theorem.
\begin{theorem}\label{pointwise.1}
Consider a~Markov process as described above, and let~$A,B$ be such that
(\ref{intro.10}) holds. Then
%
%
\begin{equation}
\label{intro.11}
\max_{\s,\h\in A} \biggl| \frac
{\E_\s\tau_B}{\E_{\h}\tau_B}-1\biggr|\leq e^{-CN/2}.
\end{equation}
\end{theorem}
\begin{remark*}
Assumptions (A1) and (A2) are formulated in the context in which we
will prove our results. The restriction to the state space $\{-1,1\}^N$ is
mainly done because we need to construct an explicit coupling. It is rather
straightforward to generalize everything to the case of Potts spins
($\cS_N\equiv\{1,\ldots,q\}^N)$ and maps $ \un m^{n}$ whose
components are
permutation invariant functions of the spin variables on~$\Lambda_i^{n}$.
\end{remark*}

The claim of Theorem~\ref{pointwise.1} is trivial whenever ${\mathbb
P}_\s(\tau_\eta<\tau_B)$ is exponentially close to one, as
$N\uparrow\infty$. However, in the context of stochastic Ising models
it is reasonable to expect that, for fixed $\s,\h\in A$, ${\mathbb
P}_\s(\tau_\eta<\tau_B)$ is exponentially small. That is, despite the
fact that a~chain starting at $\s$ spends an exponentially large amount
of time in $A$, this time is not long enough for visiting more than a~small fraction of the exponentially large number of microscopic points
in $A$. An alternative approach is to try to construct a~coupling
between $\s$ and $\h$ chains. In the case of the Curie--Weiss model
(without random fields), a~useful coupling algorithm was suggested in
the recent paper~\cite{LLP08}. This algorithm ensures that:

\begin{longlist}[(b)]
\item[(a)] If $\un m^n (\s_s ) = \un m^n (\h_s )$,
then $\un m^n (\s_t ) = \un m^n (\h_t )$ for all
$t\geq s$.
\item[(b)] The Hamming distance between $\s_t$ and $\h_t$ is
nonincreasing in time.
\end{longlist}

In a~way, this is reminiscent of the stochastic stability results of
\cite{MS88}. It is straightforward to adjust the construction of\vadjust{\goodbreak}
\cite{LLP08} to the general context we consider here. But both (a) and
(b) above would be lost, and it is not clear that such a~coupling would
work globally.

Instead, our strategy is to use (\ref{intro.10}) and to keep trying to
couple the $\s$-chain with a~typical $\h$-chain each time when $\s_t$
enters $\cS^n (\un m (\h))$. In the sequel, we call this the
\textit{basic coupling attempt}. Clearly, in view of a~possible biased
sampling, basic coupling attempts should be designed with care, which
explains the relatively complicated construction in Section
\ref{section.2.2}. It is based on~\cite{LLP08}, but we need to enlarge
the probability space in order to achieve sufficient independence
between decision making and properties of the eventually chosen
$\eta$-path. In particular, the fact that $\s $-chain and $\h$-chain
meet will not automatically imply coupling.

A second and related problem that tends to arise in the situation that
we are interested in is the breakdown of strict renewal properties.
This is a~well-known issue in the theory of continuous space Markov
processes where methods such as Nummelin splitting~\cite{nummel} were
devised to prove ergodic theorem for the Harris recurrent chains. Here
we would like to use renewal arguments, for example, to prove
asymptotic exponentiallity of the law of~$\tau_B$. We will show that
again coupling arguments can be used to solve such problems.

As an example, we will prove the following theorem.
\begin{theorem}\label{exponential.1}
In the random field Curie--Weiss model, for $A$ and $B$ chosen to
satisfy the hypothesis of Theorem~\ref{pointwise.1},
%
%
\begin{equation}\label{exponential.3}
{\mathbb P}_{\s} (\tau_B/\E_\s\tau_B >t) \rightarrow
e^{-t}\qquad\mbox{as } N\uparrow\infty
\end{equation}
for all $\s\in A$ and for all $t\in\R_+$.
\end{theorem}

Theorem~\ref{exponential.1} is proven in Section~\ref{S1}. The basic
idea is to use our iterative coupling procedure for deriving a~renewal-type equation for the Laplace transform of~$\tau_B$.

\section{The random field Curie--Weiss model}\label{Sect.2}

The results of this paper are motivated by the study of the Glauber
dynamics of the random field Curie--Weiss model (RFCW). We will show
that the assumptions of the two theorems above can be verified in that
model. In this section, we briefly recall results for this model
obtained recently in~\cite{BBI08} and prove an elementary local
recurrence estimate.

\subsection{The model and equilibrium properties}\label{sec21}

In the RFCW model, the state space is ${\mathcal S}_N\equiv\{-1,1\}
^N$, the Gibbs measure is given by
%
%
\begin{equation}\label{rfcw.1}
\mu_N(\s)=Z_N^{-1} \exp(-\beta{H}_N(\s)),
\end{equation}
and the \textit{random Hamiltonian}, $H_N$, is defined as
%
%
\begin{equation}\label{rfcw.2}
H_N(\s) \equiv-\frac N2 \biggl(\frac1N\sum_{i\in\Lambda}
\s_i\biggr)^2 -\sum_{i\in\Lambda} h_i\s_i,\vadjust{\goodbreak}
\end{equation}
where $\Lambda\equiv\{1,\ldots, N\}$ and $h_i$, $i\in\Lambda$, are
i.i.d. random variables on some probability space $(\Omega, \cF,
{\mathbb P}_h )$.

The total magnetization
%
%
\begin{equation}\label{static.3}
m_N(\s)\equiv\frac1N\sum_{i\in\Lambda}\s_i
\end{equation}
is an effective order parameter of the model, and the sets of
configurations where the magnetization takes particular values play
the r\^{o}le of \textit{metastable states}. More specifically, we
introduce the law of $m_N$ through
%
%
\begin{equation}\label{static.4}
\cQ_{\beta,N}\equiv
\mu_{\beta,N}\circ m_N^{-1}
\end{equation}
on the set of possible values $\G_N\equiv\{-1,-1+2/N,\ldots,1\}$.
$\cQ_{\beta,N}$ satisfies a~large deviation property, in particular
%
%
\begin{equation}\label{static.11.1}
Z_{\beta,N}\QQ_{\beta,N}(m)= \sqrt{{\frac{2 I''_{N}(m)}{N\pi}}}
\exp(-N\beta{F}_{\beta,N}(m)) \bigl(1+
o(1)\bigr)
\end{equation}
with $I_N$ being the Legendre transform of
%
%
\begin{equation}
t \mapsto\frac1{N}\sum_{i\in\Lambda} \log\cosh( t +\beta{h}_i
),
\end{equation}
and with an explicit form for the rate function (``free energy''),
$F_{\beta,N}$. The metastable states correspond to multiple local
minima of $F_{\beta,N}$, whenever they exist.

A crucial feature of the model is that we can introduce a~family of
\textit{mesoscopic} variables in such a~way that the dynamics on these
mesoscopic variables is well approximated by a~Markov process. Let us
briefly describe these mesoscopic variables.

\subsection{Coarse graining}
\label{SS21}
Let $I$ denote the support of the distribution of the random fields.
Let $I_\ell$, with $\ell\in\{1,\ldots,n\}$, be a~partition of $I$
such that, for some $C<\infty$ and
for all $\ell$, $|I_\ell|\leq C/n\equiv{\varepsilon}$.

Each realization of the random field $\{h_i\}_{i\in\N}$ induces
a random partition of the set $\Lambda\equiv\{1,\ldots,N\}$ into subsets
%
%
\begin{equation} \label{9.2}
\Lambda_k\equiv\{i\in\Lambda\dvtx h_i\in I_k\}.
\end{equation}
We may
introduce $n$ order parameters
%
%
\begin{equation} \label{9.3}
m_k(\s)\equiv\frac
1N\sum_{i\in\Lambda_k}\s_i.
\end{equation}
We denote by $ \vm$ the $n$-dimensional vector $(m_1,\ldots,m_n)$.
$\vm$ takes values in the set
%
%
\begin{equation} \label{9.4}
\G_N^n \equiv\bigtimes_{k=1}^n
\biggl\{-\rho_{N,k},-\rho_{N,k}+{\frac{2}{N}},\ldots,
\rho_{N,k}-{\frac{2}{N}},\rho_{N,k}\biggr\},
\end{equation}
where
%
%
\begin{equation}\label{9.5}
\rho_k\equiv\rho_{N,k}\equiv\frac{|\Lambda_k|}N.
\end{equation}
Note that the random variables $\rho_k$ concentrate exponentially (in
$N$) around their mean value $\E_h\rho_{N,k}={\mathbb P}_h[h_i\in
I_k]\equiv p_k$. The Hamiltonian can be written as
%
%
\begin{equation}\label{9.6}
H_N(\s)= -N E(\vm(\s)) +\sum_{\ell=1}^n \sum_{i\in
\Lambda_\ell}\s_i\wt h_i,
\end{equation}
where $E\dvtx\R^n\rightarrow\R$ is the function
%
%
\begin{equation} \label{9.7}
E(\un x)\equiv\frac12\Biggl(\sum_{k=1}^n x_k\Biggr)^2+\sum_{k=1}^n
\bar h_k x_k
\end{equation}
with
%
%
\begin{equation}\label{9.7.1}
\bar h_\ell\equiv
\frac1{|\Lambda_\ell|}\sum_{i\in\Lambda_\ell} h_i\quad \mbox{and}\quad
\wt h_i\equiv h_i-\bar h_\ell.
\end{equation}
The equilibrium distribution of the variables
$\vm(\s)$ is given by
%
%
\begin{eqnarray}\label{9.8}
\cQ_{\beta,N}(\un x)
&\equiv&
\mu_{\beta,N}\bigl(\vm(\s)=\un x\bigr)
\nonumber\\[-8pt]\\[-8pt]
&=& \frac1{Z_N} e^{\beta{N}E(\un x)} \E_\s
\mathbh{1}_{\{\vm(\s)=\un x\}}e^{\sum_{\ell=1}^n \sum_{i\in
\Lambda_\ell}\s_i(h_i-\bar h_\ell)}.\nonumber
\end{eqnarray}
For a~mesoscopic subset, $\bA\subseteq\G_N^n$,
we define its microscopic counterpart, $A$, as
%
%
\begin{equation}\label{put-numbers.10}
A = \cS^n[\bA] = \{\sigma\in\cS_N\dvtx\bmm(\sigma)\in\bA\}.
\end{equation}
Note that, as
in the one-dimensional case, we can express the right-hand side of
(\ref{9.8}) as
%
%
\begin{equation}\label{fine.1.1}
Z_{\beta,N}\QQ_{\beta,N}(\un x)=\prod_{\ell=1}^n
{\sqrt{\frac{(I''_{N,\ell}(x_\ell/\rho_\ell)/\rho_\ell
)}{{N\pi}/ 2}}}
\exp(-N\beta{F}_{\beta,N}(\un x)) \bigl(1+\so\bigr)\hspace*{-28pt}
\end{equation}
with an explicit expression for the function
$F_{\beta,N}$,
%
%
\begin{equation}\label{fine.16}
F_{\beta,N}(\un x)\equiv- \frac12
\Biggl(\sum_{\ell=1}^n x_\ell\Biggr)^2 -\sum_{\ell=1}^n x_\ell
\bar h_\ell+\frac1\beta\sum_{\ell=1}^n \rho_\ell
I_{N,\ell}(x_\ell/\rho_\ell).
\end{equation}
The key point of the construction above is that it places the RFCW
model in the context described in Section~\ref{subsection.1.2}. Namely,
defining the mesoscopic rates, $r_N(\un m,\un m')$, in (\ref{intro.5})
for the functions $\un m$ defined in (\ref{9.3}), one can easily verify
that the estimates (\ref {intro.6}) hold, as was exploited in
\cite{BBI08}. In the next subsection, we will show that the recurrence
hypothesis (\ref{intro.10}) also holds in this model.

In~\cite{BBI08}, we proved the following.
\begin{theorem}\label{theorem1}
Assume that $\beta$ and the distribution of the magnetic field are such
that there exist more than one local minimum of $F_{\beta,N}$. Let
$m^*$ be a~local minimum of $F_{\beta,N}$, $M\equiv M(m^*)$ be the set
of minima of $F_{\beta,N}$ such that $F_{\beta,N}(m)<F_{\beta,N}(m^*)$,
and $z^*$ be the minimax between $m$ and $M$, that is, the lower of the
highest maxima separating $m$ from $M$ to the left, respectively, right.
Then, ${\mathbb P}_h$-almost surely,
%
%
\begin{eqnarray}\label{theorem1.2}
\E_{\nu_{S[m^*],S[M]}} \tau_{S[M]}&=&
C(\beta,m^*,M) N\exp\bigl(\beta{N}[F_{\beta,N}(z^*)-F_{\beta
,N}(m^*)]\bigr)\nonumber\hspace*{-35pt}\\[-8pt]\\[-8pt]
&&{}\times\bigl(1+o(1)\bigr),\nonumber\hspace*{-35pt}
\end{eqnarray}
where $C(\beta,m^*,M)$ is a~constant that is computed explicitly in
\cite{BBI08}.
\end{theorem}

Here the initial measure, $\nu_{S[m^*],S[M]}$,
is the so-called \textit{last exit biased distribution} on the set
$S[m^*]\equiv\{\s\in\cS_N\dvtx m_N(\s)=m^*\}$, given by
the formula
%
%
\begin{equation}\label{theorem1.1}
\nu_{A,B}(\s)=\frac{\mu_{\beta,N}(\s){\mathbb P}_\s[\tau
_{B}<\tau_A]}
{\sum_{\s\in A}\mu_{\beta,N}(\s){\mathbb P}_\s[\tau_{B}<\tau_A]}.
\end{equation}

Although the theorem is stated in~\cite{BBI08} for the starting measure
in a~set defined with respect to the one-dimensional order parameter, the
estimates given there immediately imply that the same formulas hold
replacing $m^*$ with a~local minimum, $\vm^*$, in
the $n$-dimensional order parameter space.

Theorem~\ref{theorem1} implies that the estimate (\ref{theorem1.2})
holds for $\E_\s\tau_{S[M]}$ for any $\s$ in a~neighborhood of
$\cS[\vm^*]$, for $n$ large enough.

\subsection{Local recurrence}\label{sec23}

Before starting the proof of (\ref{intro.11}), let us verify that the
hypothesis (\ref{intro.10}) holds for the RFCW model. Specifically, let
us define the metastable set $\bA_\delta\subset\G_n$ as the ball, with
respect to the Hamming distance, of fixed radius $\delta{N}$,
$\delta>0$, centered on a~local minimum $\bmm^*$ of $F_{\beta,N}$. Let
$A_\delta\subset\cS_N$ be the corresponding microscopic metastable set
and denote by $\tau_{\bmm}$ the first hitting time of the set
$\cS^n[\bmm]$. With this notation, we have the following lemma.
\begin{lemma}\label{lemma:hypothesis}
There exist $\delta>0$ and $c_1>0$ such that, for all $n$ large enough,
$\s, \s'\in A_\delta$,
%
%
\begin{equation}\label{hypothesis.1}
{\mathbb P}_{\s}\bigl[\tau_{B}<\tau_{\bmm(\s')}\bigr]\leq
e^{ -c_1N}.
\end{equation}
\end{lemma}
\begin{pf}
We first notice that if $\s'\in\cS^n[\bmm^*]$, then the assertion
of the
lemma
holds for all $n$ sufficiently large with a~constant $c_0$
independent of $n$, as has been proven in~\cite{BBI08} (see Proposition 6.12).

Moreover, for all $\s,\s'\in A_\delta$,
%
%
\begin{equation}\label{hypothesis.2}
{\mathbb P}_{\s}\bigl[\tau_{\bmm(\s')}<\tau_{\un m^*}\bigr]\geq
e^{ -c\delta{N}}
\end{equation}
for some positive constant $c$.
To see this, notice that, due to the property of~$\bA_\delta$,
one can find a~mesoscopic path
from $\bmm(\s)$ to $\bmm(\s')$ with length at most~$\delta{N}$.
Implementing the argument that is used in the proof of Lemma~6.11
of~\cite{BBI08}, one gets~(\ref{hypothesis.2}).

To prove (\ref{hypothesis.1}), we use a~renewal argument. Let us
consider a~configuration $\s\in\cS^n[\bmm^*]$ and a~generic $\s'\in
A_\delta$, and set $\bmm\equiv\bmm(\s')$. Then
%
%
\begin{eqnarray}\label{renewal.1}\qquad
{\mathbb P}_{\s}(\tau_{B}<\tau_{\bmm})&\leq&
{\mathbb P}_{\s}(\tau_{B}<\tau_{\bmm}\wedge\tau_{\bmm^*}) +
{\mathbb P}_{\s}(\tau_{\bmm^*}<\tau_{B}<\tau_{\bmm})
\nonumber\\
&\leq&
{\mathbb P}_{\s}(\tau_{B}<\tau_{\bmm^*}) + \max_{\eta\in\cS
^n[\bmm^*]}
{\mathbb P}_\eta(\tau_{B}<\tau_{\bmm}) {\mathbb P}_\s(\tau_{\bmm
^*}<\tau_{\un m})
\\
&\leq& e^{-c_0 N} +
\max_{\eta\in\cS^n[\bmm^*]}{\mathbb P}_\eta(\tau_{B}<\tau_{\bmm
}) (1-e^{-c\delta{N} }),\nonumber
\end{eqnarray}
where in the second line we used the Markov property, and in the last
line we used the inequality (\ref{hypothesis.1}) and
(\ref{hypothesis.2}). Taking the maximum over $\s\in\cS^n[\bmm^*]$ on
both sides of (\ref {renewal.1}) and rearranging the summation, we get
the inequa\-lity~(\ref{hypothesis.1}) for $\s\in\cS^n[\bmm^*]$, with a~constant $c_1=c_0-c\delta$ which is strictly positive for small enough
$\delta$.

Now let us consider the general case when $\s, \s'\in A_\delta$ and
set again
$\bmm\equiv\bmm(\s')$. As before, we have
%
%
\begin{eqnarray}\label{renewal.2}\qquad
{\mathbb P}_{\s}(\tau_{B}<\tau_{\bmm})&\leq&
{\mathbb P}_{\s}(\tau_{B}<\tau_{\bmm}\wedge\tau_{\bmm^*}) +
{\mathbb P}_{\s}(\tau_{\bmm^*}<\tau_{B}<\tau_{\bmm})
\nonumber\\
&\leq&
{\mathbb P}_{\s}(\tau_{B}<\tau_{\bmm^*}) +
\max_{\eta\in\cS^n[\bmm^*]}
{\mathbb P}_\eta(\tau_{B}<\tau_{\bmm}) {\mathbb P}_\s(\tau_{\bmm
^*}<\tau_{B})
\\
&\leq& e^{-c_0N}+ e^{-c_1N}
=e^{-c_1N}\bigl(1+o(1)\bigr),\nonumber
\end{eqnarray}
where in the third line we used the fact that the inequality (\ref
{hypothesis.1})
was already established
for $\h\in\cS^n[\bmm^*]$.
This concludes the proof of the lemma.
\end{pf}

Lemma~\ref{lemma:hypothesis} shows that, for all $n$ large enough,
the RFCW model parameterized by the variables $\bmm\in\G_N^n$
satisfies the hypothesis (\ref{intro.10}) with $A=A_\delta$ as
defined above.

\section{Construction of the coupling}\label{S1}
\subsection{The coupling by Levin, Luczak and
Peres} \label{subsection.2.1}

Recall that we consider a~partition, $\{\Lambda_1,\ldots, \Lambda_n\}$,
of $\Lambda \equiv\{1,\ldots, N\}$ and let $\un m = (m_1(\s),\ldots,
m_n(\s))$ be the vector of partial magnetizations as defined in
(\ref{9.3}).

We begin by explaining a~coupling that was used by Levin, Luczak and
Peres~\cite{LLP08} in the usual Curie--Weiss model. In that case, the
transition rates have the following properties: whenever $x, y$ and
$\s,\h$ are such that:
\begin{longlist}[(ii)]
\item[(i)] $\un m (\s)=\un m(\h)$, and
\item[(ii)] $\s_x = \h_y $,\vadjust{\goodbreak}
\end{longlist}
then
%
%
\begin{equation}\label{2.3}
p_N\bigl(\s,\s^{(x)}\bigr) =p_N\bigl(\h,\h^{(y)}\bigr).
\end{equation}
%
We continue to employ
the notation
$p_x^{\pm}(\s)$,
%
%
\begin{equation}\label{peeplus}
p_{x}^{-\s_x } (\s) \equiv N
p_N\bigl(\s, \s^{(x)}\bigr)
\quad\mbox{and}\quad p_{x}^{+} (\s) + p_{x}^{-} (\s) = 1,
\end{equation}
where, as usual, $\s^{(x )}$ is the configuration obtained from $\s$ by
setting $\s^x_x=-\s_x$ and leaving all other components of $\s$
unchanged.

The coupling of Levin, Luczak and Peres is constructed as follows. Let
$\s$ and $\h$ be two initial conditions such that $\un m(\s)=\un
m(\h)$. Let $I_t$, $t=0,1,2,\ldots,$ be a~family of independent random
variables that are uniformly distributed on $\Lambda$. Assume that at
time $t$, $\un m(\s(t))=\un m(\h(t))$ and do the following:
\begin{enumerate}[(O2)]
\item[(O1)] Draw the random variable $I_t$;
\item[(O2)] Set $\eta_{I_t}(t+1)= \pm1 $ with probabilities
$p_{I_t}^{\pm}(\eta(t))$, respectively,
and set $\eta_x(t+1)=\eta_x(t)$ for all $x\neq I_t$;
\item[(A)] Then do the following:
\begin{enumerate}[(ii)]
\item[(i)] If $\s_{I_t}(t)=\h_{I_t}(t)$,
then set:
\begin{enumerate}[$*$]
\item[$*$]
$\s_{I_t}(t+1)=\h_{I_t}(t+1)$;
\item[$*$] $\s_x(t+1)=\s_x(t)$, for all $x\neq I_t$.
\end{enumerate}
\item[(ii)] If $\s_{I_t}(t)\neq\h_{I_t}(t)$, then
let $\Lambda_k$ be the element of the partition such that $I_t\in
\Lambda_k$ and choose $y$ uniformly at random on the set
$\{z\in\Lambda_k\dvtx\s_z(t)\neq\h_z(t)\neq
\h_{I_t}(t)\}$.
Note that this set is not empty, since $\un
m(\s(t))=\un m(\h(t))$ and $\s(t)$ and $\h(t)$ differ in
one site of $\Lambda_k$. Then set:
\begin{enumerate}[$*$]
\item[$*$] $\s_{y}(t+1)=\h_{I_t }(t+1 )$;
\item[$*$] $\s_x(t+1)=\s_x(t)$, for all $x\neq y$.
\end{enumerate}
\end{enumerate}
\end{enumerate}
Note that this construction has the virtue that $\un m(\s(t+1))
=\un m(\h(t+1))$, so that the assumption inherent in the
construction is always verified, if it is verified at time zero.

Moreover, if $\s(t)=\h(t)$ for some $t$, then $\s(t+s)=\h(t+s)$, for
all $s\geq0$. Finally, one easily checks that the marginal
distributions of $\s(t)$ and $\h(t)$ coincide and are given by the law
of the original dynamics. This latter fact depends crucially on the
fact that the flip rates do not depend on which site in a~given subset
$\Lambda_i$ the spin is flipped, provided they are flipped in the same
direction.

\subsection{Coupling attempt in the general case}\label{section.2.2}

In the general case, we consider here, including the RFCW, (\ref{2.3})
does not hold unless $x = y$. All we assume is (A.2) and
(\ref{intro.6}). The problem is now that the probabilities to {update}
the $\s$-chain in a~chosen point $y$ are typically not the same as
those of the $\h$-chain in the original point $I_t$. However, by
(\ref{intro.6}), these probabilities are still close to each other, in
the sense that there exists $\nu= \nu(n)$ with $\nu\downarrow0$ as
$n\uparrow\infty$, for example, $\nu(n ) = 3\epsilon(n)$, such that for
any $k$, for any $x ,y\in\Lambda_k$ and for any $\s$ and $\h$ with $\un
m(\s) = \un m (\h)$ and $\s_x =\h_y$,
%
%
\begin{equation}
\label{eq.ration}
\frac{p^{\pm}_x (\h)}{p^\pm_y (\s)} \leq1 +\nu.
\end{equation}

Thus, in order to maintain the correct marginal distribution for the
processes, we have to change the updating rules in such a~way that the
$\s$-chain will sometimes not maintain the same magnetization as the
$\h$-chain, which implies that the coupling cannot be continued.

The basic strategy to overcome this difficulty is to use iterated
coupling attempts. We shall decompose the $\s$-path on $[0, \tau_B^\s)$
into cycles and during each cycle we shall attempt to couple it with an
independent copy of the $\h$-chain. In the case of success, both chains
will run together until $\tau_B$. Such procedure necessarily involves a~sampling of $\eta$-paths. In order to control its bias, it will be
important to separate the path properties of $\eta$-chains with which
we try to couple from the probability of whether a~subsequent coupling
attempt is successful or not. This will be achieved by constructing a~coupling on an extended probability space.

\subsubsection*{Basic coupling attempt}
There are two parameters $c_2 >0$ and $\kappa<\infty$ whose
values will be quantified in the sequel.

Let $\eta$ and $\s$ satisfy $\un m (\eta) = \un m (\s)$. We shall try
to couple a~$\s$-path with an $\eta$-path during the first
$N^\kappa$-steps of their life. Let $M = c_2 N$ and let $V_i, i = 1,
\ldots, M$, be a~family of i.i.d. Bernoulli random variables with
%
%
\begin{equation}\label{2.5}
{\mathbb P}[V_i=1]=1-{\mathbb P}[V_i=0] = 1-
\nu(n).
\end{equation}
We now describe how the coupling construction is adjusted using the
random variables $V_i, i=1, \ldots,M$. Let $\un m (\eta(0) ) = \un m
(\s(0) )$.

As before, let $I_t$, $t=0,1,2,\ldots,N^\kappa$, be a~family of
independent random variables that are uniformly distributed on
$\Lambda$. Let $\cM_0=0$ and $\chi_0 =0$, $\eta(0)=\eta$, $\s(0)=\s$.
At time $t\geq1$, do the following:
\begin{enumerate}[(O2)]
\item[(O1)] Draw the random variable $I_t$;
\item[(O2)] Set $\eta_{I_t}(t+1)= \pm1$ with probability
$p_{I_t}^{\pm}(\eta_{I_t}(t))$ and set $\eta_x(t+1)=\eta_x(t)$ for
all $x\neq I_t$;
\item[(A)] If at time $1\leq t\leq N^\kappa$,
$\chi_t = 0$ and $\cM_t <M$, then
do the following:
\begin{enumerate}[(ii)]
\item[(i)]
If $\s_{I_t}(t)=\h_{I_t}(t)$, then set:
\begin{enumerate}[$*$]
\item[$*$] $\s_{I_t}(t+1)=\h_{I_t}(t+1)$;
\item[$*$] $\s_x(t+1)=\s_x(t)$, for all $x\neq I_t$;
\item[$*$] $\cM_{t+1}=\cM_t$.
\end{enumerate}
\item[(ii)] If $\s_{I_t}(t)\neq\h_{I_t}(t)$,
let $\Lambda_k$ be the element of the partition such that $I_t\in
\Lambda_k$ and, as before, choose $y$ uniformly at random on the set
$\{z\in\Lambda_k\dvtx\s_z(t)\neq\h_z(t)\neq
\h_{I_t}(t)\}$. Then set:
\begin{itemize}[$*$]
\item[$*$] $\s_{y}(t+1)=\h_{I_t}(t+1)$ with probability
%
%
\begin{equation}\label{tt.31}
\cases{
1, &\quad if $V_{\cM_t}=1$,\cr
\displaystyle \frac{p_{I_t}^{\pm}(\h(t) )\wedge p_{y}^{\pm}(\s(t) )
- (1-\nu)p_{I_t}^\pm(\h(t) )}
{\nu p_{I_t}^{\pm}(\h(t))},
&\quad if $V_{\cM_t}=0$,}
\end{equation}
and
$\s_{y}(t+1)=-\h_{I_t}(t+1)$ with probability
%
%
\begin{equation}\label{tt.44}
\cases{
0, &\quad if $V_{\cM_t}=1$,\cr
\displaystyle \frac{ p_{I_t}^{\pm}(\h(t)) - p_{I_t}^\pm(\h(t))\wedge p_y^{\pm
}(\s(t))}
{\nu p_{I_t}^{\pm}(\h(t))}, &\quad if $V_{\cM_t}=0$;}
\end{equation}
\item[$*$] $\s_x(t+1)=\s_x(t)$, for all $x\neq y$;
\item[$*$] If $V_{\cM_t}=0$, then set $\chi_s=1$ for $s= t+1, \ldots,
N^\kappa$, otherwise set $\chi_{t+1}=\chi_{t}$;
\item[$*$]
Set $\MM_{t+1}=\cM_t+1$;
\end{itemize}
\end{enumerate}
\item[(B)]If at time $t$,
either $\chi_t =1$ or $\cM_t =M$,
then update $\s$ independently of $\eta$, that is:
\begin{enumerate}[(ii)]
\item[(i)] draw $I'_t$ independently with the same law as $I_t$, and
\item[(ii)] set
$\s_{I'_t}(t+1) = \pm1$ with probability $p_{I'_t}^{\pm}(\s(t))$,
and $\s_x(t+1)=
\s_x(t)$, for all $x\neq I'_t$.
\end{enumerate}
\end{enumerate}
The process $\cM_t$ is a~counter that increases by one each time a~new
coin $V_i$ is used by the coupling. The value $\chi_t =1$ of the
variable $\chi_t$ indicates that a~zero coin was used by time $t$.

The following lemma collects the basic properties of the process constructed
above.
\begin{lemma}\label{properties.1} Let $\wt{\mathbb P}$ denote the
joint distribution of
the processes $\s,\eta,V$ defined above. Then the above is a~good
coupling in
the sense that the marginal distributions of both $\eta(t), t
\leq N^\kappa$, and
$\s(t), t
\leq N^\kappa$, under the law $\wt{\mathbb P}$ are ${\mathbb P}_{\s
(0)}$ and ${\mathbb P}_{\eta(0)}$, respectively.
\end{lemma}
\begin{pf} The assertion is obvious for the process $\eta(t)$.
It is also clear for the $\s(t )$ process if updates are done according
to case B. Therefore, we only need to check that it holds for process
$\s(t)$ at such times $t\leq N^\kappa$ when $\chi_t$ is still $0$ and
$\cM_t$ is still less than $M=c_2 N$.
In other words, we have to compute
%
%
\begin{equation}\label{coup.1}
\wt{\mathbb P}[\s(t+1)=\s_x^+ (t) |\s(t); \chi_t = 0 ; \cM
_t <c_2 N],
\end{equation}
where $\s_x^+ \df(\s_1,\ldots,\s_{x-1},+1,\ldots,\s_N )$.
First, it is clear that, given that $I_{t}=x$ and $\s_{I_{t}}(t)
=\eta_{I_{t}}(t)$,
we get the desired result, that is,
%
%
\begin{eqnarray}\label{coup.2}
&&
\wt{\mathbb P}\bigl[\s(t+1)= \s_x^+ (t)
| I_{t}=x;
\s_x(t)=\eta_x(t) ; \s(t) ;\chi_t = 0 ; \cM_t <c_2 N
\bigr)\nonumber\\[-8pt]\\[-8pt]
&&\qquad= p_x^+(\eta(t))=p_x^+(\s(t) ).\nonumber
\end{eqnarray}
In the case $I_{t}=y\neq x$, we get a~contribution
to (\ref{coup.1}) only if:
\begin{longlist}[(iii)]
\item[(i)] $x,y$ are in the same set $\Lambda_i$,
\item[(ii)] $\s_y(t)\neq\h_y(t)$,
\item[(iii)] $\s_x(t ) \neq\h_x(t)$, and
\item[(iv)] $\s_x(t)= \h_y(t)$.
\end{longlist}
If these conditions are satisfied, the probability to flip $\s_x$ to
$+1$ is
%
%
\begin{eqnarray}\label{coup.3}\quad
&&(1-\nu) p_y^+(\eta(t))+\nu p_x^+(\eta(t))\frac{ p_{y}^+(\h
(t))\wedge p_x^{+}(\s(t)) -(1-\nu)p^+_y(\eta(t))}
{{\nu}p_{y}^{+}(\h(t))}\nonumber\\
&&\quad{}+\nu p^-_y\frac{ p_{y}^{-}(\h(t)) - p_{y}^-(\h(t))\wedge
p_x^{-}(\s(t))}
{{\nu}p_{y}^{-}(\h(t))}\nonumber\\[-8pt]\\[-8pt]
&&\qquad=
p_{y}^+(\h(t))\wedge p_x^{+}(\s(t))
+ p_{y}^{-}(\h(t)) - p_{y}^-(\h(t))\wedge p_x^{-}(\s(t))\nonumber\\
&&\qquad=
p_x^+(\s(t)).\nonumber
\end{eqnarray}
The last line is easily verified by distinguishing cases.
It follows that the probability in (\ref{coup.1}) is equal to
$N^{-1} p_x^+(\s(t))$, as desired. This proves the lemma.
\end{pf}

The construction above tries to merge the processes $\eta(t)$ and
$\s(t)$ only as long as $\chi_t=0$ and $\MM_t<M$. Note that if for some
$t < N^\kappa$ both these conditions still hold and, in addition,
$\eta(t)=\s(t)$, then the two dynamics automatically stay together
until $N^\kappa$ and, indeed, stay coupled forever. Naively, one would
want to classify such situation as ``successful coupling.'' However,
this would involve an implicit sampling of $\eta$-trajectories which
may lead to distortion of their statistical properties. For example, it
is not clear whether the correct value of $\E_\h\tau _B$ would survive
such a~procedure.

In order to circumvent this obstacle, we use a~more restrictive
definition of what a~``successful coupling'' should be. Namely, we say
that our basic coupling attempt is successful if the following two
\textit{independent} events $\cA$ and~$\cB$ simultaneously happen on
the enlarged probability space:
\begin{longlist}[(2)]
\item[(1)] The event
%
%
\begin{equation}\label{eventa.1}
\cA\equiv\Biggl\{ \bigforall V_i=1\Biggr\}
\end{equation}
is the event that all $M$ random variables
$V_i$ should be equal to $1$.
\item[(2)]
The event $\BB$ depends only on the random variables
$\eta(t), t\leq N^\kk$. To define it, we introduce
two stopping times, $S$ and $\NN$. Let
%
%
\begin{equation}\label{event.2}
S_x= \inf\{ t\dvtx\h_x (t+1) = -\h_x (0 ) \}
\end{equation}
and set $S\equiv\max_{1\leq x\leq N} S_x$.
Clearly, $S_x$ is the first time the spin
at site $x$ has been flipped and $S$ is the first time\vadjust{\goodbreak}
all coordinates of $\eta$ have been flipped. $\NN$ is
defined as
%
%
\begin{equation}\label{event.4}
\NN\equiv\sum_{t=0}^{S}\sum_{x=1}^N
\mathbh{1}_{\{I_t = x\}}\mathbh{1}_{\{t\leq S_x\}},
\end{equation}
which is the total number of flipping attempts until time $S$.
Finally,
%
%
\begin{equation}\label{event.5}
\BB\equiv\{\tau^{\h}_B \geq N^\kappa\}\cap\{
S < N^\kappa\}\cap\{\NN\leq M
\}.
\end{equation}
\end{longlist}

The important observation is the following.
\begin{lemma}\label{ab.1}
On ${\mathcal A}\cap\BB$, the coupling is successful in the sense that
%
%
\begin{equation}\label{success.1}
{\mathcal A}\cap\BB\subset\{\h(N^\kk)=\s(N^\kk)\}.
\end{equation}
\end{lemma}
\begin{pf}
On the event $\BB\cap{\mathcal A}$, by time $N^k$ $\eta(t)$ has not
reached $B$, all
spins have been flipped once, and each flip that involved a~site where
$\eta(t)\neq\s(t)$
was done when the coin $V_i$ took the value $+1$. Therefore, on
each first flip the corresponding $\eta$ and $\s$ spins became aligned,
hence $\eta(N^\kk)=\s(N^\kk)$.
\end{pf}
\begin{remark*}
Note that the inclusion (\ref{success.1}) is in general strict.
The rationale for the introduction of the events ${\mathcal A}$ and
$\BB$ is that
the unlikely event ${\mathcal A}$ does not affect the $\eta$-chain at
all and that the
(likely) event $\BB$ does not
distort the hitting times of the $\eta$-chain in the sense that
$\E_{\h}(\tau_B \mathbh{1}_{\BB})\geq\E_\h\tau_B(1
- e^{-c N} )$.
This will be part of the content of Lemma~\ref{event.8} which we formulate
and prove below.
\end{remark*}

\subsection{\texorpdfstring{Construction of a~cycle and cycle decomposition of $\s$-paths}
{Construction of a~cycle and cycle decomposition of sigma-paths}}
We have seen that $\BB\cap{\mathcal A}$ indicates that our coupling
is successful and
$\eta(t)$ and $\s(t)$ arrive together in $B$. However, the
probability of
${\mathcal A}\cap\BB$
is very small, essentially due to the fact that the probability of
${\mathcal A}$ is
small, namely ${\mathbb P}({\mathcal A})=(1-\nu)^M$. What will be
essential is that the
probability of $\BB$ is otherwise close to one, and therefore the
$\eta$-paths (which are independent of the~$V_i$)
will be affected very little by the occurrence of ${\mathcal A}\cap\BB$.

We then have to decide what to do on $({\mathcal A}\cap\BB)^c$ at
time $N^\kk$.
Define the stopping time
%
%
\begin{equation}\label{delta.1}
\Delta= \min\{ t> N^\kappa\dvtx\s(t)\in\mathcal S^n[\un m (\h
)]\}.
\end{equation}
If
%
%
\begin{equation}\label{delta.2}
\DD= \{\Delta<\tau_B^\s\}
\end{equation}
happens, then we initiate a~new basic coupling attempt at time $\Delta
$ for
a new, independent
copy of the $\h$-chain and a~chain starting from $\s(\Delta)$.
Otherwise, on the event $\DD^c
\cap({\mathcal A} \cap\BB)^c $, the process stops and coupling has
not occurred.

The cycle decomposition of $\s[0, \tau_B^\s)$ is based on a~collection
$\{\h^\ell[0, \tau_B^{\ell,\h})\}$ of independent copies of
$\eta$-chains and on a~collection
$\{\un V^\ell= ( V^\ell_0 , \ldots, V^\ell_{N^\kappa})\}
$ of
i.i.d. stacks of coins.
The events $\{{\mathcal A}^\ell, \BB^\ell\}$ are well defined
and independent. The events\vadjust{\goodbreak}
$\{\DD^\ell\}$ are defined iteratively as follows: the event
$\DD^0$ is simply\vspace*{1pt} the above
event $\DD$ defined with respect the coupling attempt based on $\{
\eta^0 , \un V^0\}$.
If $\DD^0$ occurs, we denote by $\theta_0 = \Delta_0$ the random time
at which the first cycle ends.
Assume now that $\bigcap_0^{k-1}\DD^\ell\cap\bigcap(\cA^\ell\cap\cB
^\ell)^c $ happened\vspace*{1pt}
and that the $(k-1)$st cycle was finished at
a random time $\theta_{k-1}$ and at some random point $\s(\theta
_{k-1}) \in
\mathcal S^n[\un m (\h)]$. Let us\vspace*{1pt} initiate a~new basic coupling
attempt using a~new independent copy\vspace*{1pt}
$\{\h^k , \un V^k\}$ for a~chain starting at $\h$ and a~chain
starting from $\s(\theta_{k-1} )$.
The event $\cD^k$, and accordingly the cycle length $\Delta_k$, are
then defined appropriately. If $\DD^k$ happens and
$\bigcap(\cA^\ell\cap\cB^\ell)^c $
then $\theta_k\equiv\theta_{k-1} +\Delta_k$,
$\s(\theta_k ) \in\mathcal S^n[\un m (\h)]$ is well defined as
well, and the iterative procedure goes on.

In   light of the above definitions, the
enlarged probability space $\wt\Omega$ has the following disjoint
decomposition:
%
%
\begin{eqnarray}\label{decomp.1}
\mathbh{1} &=&
\sum_{k=0}^\infty\mathbh{1}_{{\mathcal A}^k}\mathbh{1}_{\BB
^k}\prod_{\ell=0}^{k-1}(
1-\mathbh{1}_{{\mathcal A}^\ell}\mathbh{1}_{\BB^\ell})\mathbh
{1}_{\DD^\ell}
\nonumber\\[-8pt]\\[-8pt]
&&{}+
\sum_{k=0}^\infty(1- \mathbh{1}_{{\mathcal A}^k}\mathbh{1}_{\BB
^k})(1- \mathbh{1}_{\DD^k})
\prod_{\ell=0}^{k-1}(
1-\mathbh{1}_{{\mathcal A}^\ell}\mathbh{1}_{\BB^\ell})\mathbh
{1}_{\DD^\ell}.\nonumber
\end{eqnarray}
As a~consequence, we arrive at the following decomposition of the
hitting time $\tau_B^\s$ in terms
of the (independent ) hitting times $\{\tau_B^{k, \h}\}$:
%
%
\begin{eqnarray}
\label{eq:tauBDecomp}
\tau_B^\s&=& \sum_{k=0}^\infty
\Biggl\{\prod_{\ell=0}^{k-1}\mathbh{1}_{\DD^{\ell}}(1- \mathbh
{1}_{{\mathcal A}^{\ell}}\mathbh{1}_{\BB^{\ell}})\Biggr\}
(\theta_{k-1} +
\tau_B^{k,\h} )\mathbh{1}_{{\mathcal A}^{k}}\mathbh{1}_{\BB
^{k}}\nonumber\\[-8pt]\\[-8pt]
&&{}+
\tau_B^\s\sum_{k=0}^\infty
\Biggl\{\prod_{\ell=0}^{k-1}\mathbh{1}_{\DD^{\ell}}(1- \mathbh
{1}_{{\mathcal A}^{\ell}}
\mathbh{1}_{\BB^{\ell}})\Biggr\}
(1- \mathbh{1}_{\DD^{k}})(1- \mathbh{1}_{{\mathcal
A}^{k}}\mathbh{1}_{\BB^{k}}).
\nonumber
\end{eqnarray}
(In both formulas above, we use the convention that products with a~negative number of terms are
equal to $1$ and set $\theta_{-1}\equiv0$.) Note that the first terms
in (\ref{decomp.1}) and
(\ref{eq:tauBDecomp}) correspond to the
cases when the iterative coupling eventually succeeds, whereas the
second term
corresponds to the case when it eventually fails.

\subsection{\texorpdfstring{Upper bounds on probabilities and proof of Theorem \protect\ref{pointwise.1}}
{Upper bounds on probabilities and proof of Theorem 1.1}}

\begin{lemma}\label{event.8}
The following estimates hold uniformly in $\s, \h\in A$:
\begin{longlist}[(ii)]
\item[(i)] There is a~constant $c>0$, independent of $n$, such that,
for $N$ large enough,
%
%
\begin{equation}\label{event.10}
{\mathbb P}(\BB^c)\leq e^{-c N}
\end{equation}
and
%
%
\begin{equation}
\label{eq:BExp}
\E_{\h}(\tau_B \mathbh{1}_{\BB})\geq\E_\h\tau_B(1
- e^{-c N} ).
\end{equation}
\item[(ii)] If $N$ is large enough,
%
%
\begin{equation}\label{event.11}
{\mathbb P}_\s(\cD)\geq1-
e^{-c N} .
\end{equation}
\end{longlist}
\end{lemma}
\begin{pf}
Item (ii) follows from Lemma~\ref{lemma:hypothesis} with, for example,
$c = c_1/2$.
To prove item (i), we write $\BB^c = \{\tau_B^\h\leq N^\kappa
\}\cup\{ S\geq N^\kappa
\}\cup\{\NN> M\}$. Thus,
%
%
\begin{equation}\label{proof.20}
\mathbh{1}_{\BB^c}=\mathbh{1}_{\{\tau_B^\h\leq N^\kk\}}+
\mathbh{1}_{\{\tau_B^\h> N^\kk\}}\mathbh{1}_{\{ S\geq N^\kk\}}
+\mathbh{1}_{\{\tau_B^\h> N^\kk\}}\mathbh{1}_{\{ S< N^\kk\}}
\mathbh{1}_{\{ \cN>M\}}.\hspace*{-35pt}
\end{equation}
Inserting this into (\ref{event.10}) and
(\ref{eq:BExp}), there are three terms to bound.
The first term is easy:
%
%
\begin{equation}\label{tra.1}
{\mathbb P}_{\h} (\tau_B \leq N^\kappa)
\leq
N^{\kk}\max_{\s'\dvtx\un m(\s')= \un m}{\mathbb P}_{\s'}(\tau
_B<\tau_{\un m})
\leq N^{\kk} e^{-c_1 N}.
\end{equation}
The first inequality used the fact that in
order to reach $B$, the process has to make one final excursion to $B$
without return to the starting set $\un m$, and that there are at most
$N^\kappa$ attempts to do so. The last inequality uses
(\ref{hypothesis.1}).
The corresponding term for (\ref{eq:BExp}) is
%
%
\begin{equation}\label{proof.1}
\E_{\h} \bigl( \tau_B \mathbh{1}_{\{\tau_B^\h< N^\kappa\}
}\bigr)
\leq N^{2\kk} e^{-c_1 N}.
\end{equation}
The second term is also easy:
first,
%
%
\begin{equation}\label{tra.3}
{\mathbb P}_\h(\{\tau_B> N^\kk\}
\cap
\{S \geq N^\kappa\} )
\leq{\mathbb P}_{\h}( S \geq N^\kappa)
\end{equation}
and
%
%
\begin{eqnarray}\label{proof.2}\qquad
\E_{\h} \bigl(\tau_B\mathbh{1}_{\{\tau_B> N^\kk\}} \mathbh{1}_{\{S
\geq N^\kappa\}} \bigr)
&\leq&
\sum_{\s'} ( N^\kappa+\E_{\s'} \tau_B )
{\mathbb P}_{\h} (\h_{N^\kappa} =\s' ; S\geq N^k)
\nonumber\\[-8pt]\\[-8pt]
&\leq&
\Bigl( N^\kappa+\max_{\s'} \E_{\s'} \tau_B\Bigr)
{\mathbb P}_{\h}( S \geq N^\kappa).\nonumber
\end{eqnarray}
Using the formula (\ref{intro.2}) with $A=\{\s'\}$, and bounding the
corresponding capacity $\capa(\s',B)\geq e^{-c_3N}$ from below
in the crudest way
(e.g., retaining a~single one-dimensional path from $\s'$ to $B$; see
\cite{Bo5}), one gets that
%
%
\begin{equation}
\label{eq:Crude}
N^\kappa+ \max_{\s'} \E_{\s'}\tau_B \leq e^{2 c_3 N} ,
\end{equation}
where $c_3$ \textit{does not} depend on $n$.

Next, we show that if $\kappa>2$ the probability
${\mathbb P}_{\h}( S \geq N^\kappa)$ is super-exponen\-tially small.
Indeed, since at each step
the probability to flip each particular spin is bounded from below by
$(1- \alpha) /N$,
%
%
\begin{equation}\label{proof.5}
{\mathbb P}_\h( S \geq N^\kappa)\leq
N
\biggl(1- \frac{1- \alpha}{N}\biggr)^{N^\kappa}\leq e^{-c_4N^{\kk
-1}} .
\end{equation}
Finally, even the third term is easy:
%
%
\begin{equation}\label{tra.5}
\E_{\h}\bigl(\mathbh{1}_{\{\tau_B> N^\kk\}}\mathbh{1}_{\{ S< N^\kk
\}} \mathbh{1}_{\{\NN> M\}}\bigr)
\leq{\mathbb P}_{\h}(\NN> M)
\end{equation}
and, as in (\ref{proof.2}),
%
%
\begin{eqnarray}\label{proof.6}
\E_{\h}\bigl(\tau_B\mathbh{1}_{\{\tau_B> N^\kk\}}\mathbh{1}_{\{ S<
N^\kk\}} \mathbh{1}_{\{\NN> M\}}\bigr)
&\leq&
\Bigl( N^\kappa+ \max_{\s'}\E_{\s'} \tau_B \Bigr){\mathbb P}_{\h}(
\NN> M)\nonumber\hspace*{-35pt}\\[-8pt]\\[-8pt]
& \leq&
e^{2 c_3 N} {\mathbb P}_{\h}(\NN> M).\nonumber\hspace*{-35pt}
\end{eqnarray}
It remains to bound ${\mathbb P}_{\h}(\NN> M)$.
In order to do this, we split the time interval $[0, S]$ into epochs
%
%
\begin{equation}\label{proof.7}
[0, S] = [0, S_{i_1}]\cup(S_{i_1 } , S_{i_2}]\cup\cdots\cup
(S_{i_{N-1}} , S] ,
\end{equation}
where $\underline{i} = \{ i_1 , \ldots, i_N\}$ is a~permutation of
$\{1, \ldots,N\}$ which is fixed by the order in which spins are
flipped for the first time,
%
%
\begin{equation}
\label{eq:Order}
S_{i_1} < S_{i_2} <\cdots< S_{i_N} =S .
\end{equation}
Fix a~particular permutation $\underline{i}$ and let
$\EE[\underline{i}]$ be the event
that (\ref{eq:Order}) happens. Let us first derive a~lower bound on
${\mathbb P}_{\h}(\EE[i ])$.
It is convenient to decompose
$\EE[\ui] =\bigcap_{k=0}^{N-1}\EE_k [\ui]$, where
%
%
\begin{eqnarray}\label{proof.8}
\EE_0 [\ui] &=&\{\mbox{No spin was flipped on $[0,
S_{i_1}-1)$}\}
\nonumber\\[-8pt]\\[-8pt]
&&{} \cap\{\mbox{Spin $i_1$ was flipped
on $S_{i_1}$-$t$ step}\}\nonumber
\end{eqnarray}
and
%
%
\begin{eqnarray}\label{proof.9}
\EE_k [\ui] &=&\{\mbox{No spin was flipped for the first time during
$[S_{i_k}, S_{i_{k+1}}-1)$}\}
\nonumber\hspace*{-35pt}\\[-8pt]\\[-8pt]
&&{} \cap\{\mbox{Spin $i_{k+1}$ was flipped
on $S_{i_{k+1}}$-$t$ step}\}.\nonumber\hspace*{-35pt}
\end{eqnarray}
Let $\NN_k$ be\vspace*{1pt} the number of times previously unflipped
spins were attempted to flip during the interval $(S_{i_k}, S_{i_{k+1}}]$.
Clearly $\NN= \sum_{k=0}^{N-1}\NN_k$.

In view of (\ref{eq:PUpdate}),
%
%
\begin{equation}
\label{eqE0}
{\mathbb P}_{\h}(\EE_0 [\ui]; \NN_0 = \ell_0 )\leq
\frac{\alpha^{\ell_0}}{N} .
\end{equation}
To give an upper bound on the probability of the events $\{\EE_k [\ui]
;\NN_k =\ell_k\}$, for $k >0$, we distinguish between two types of
trials, which happen during the intervals $ (S_{i_k }, S_{i_{k+1}})$.
First, one might choose yet unflipped spins from $\{ i_{k+1},
\ldots,i_N\}$ but then fail to flip them. On the event $\{\NN_k
=\ell_k\}$ this happens exactly $\ell_k -1$ times. Second, one might
choose already flipped spins from the set $\{ i_1 , \ldots, i_k\}$. The
probability of the latter is $k/N$, whereas, according to
(\ref{eq:PUpdate}), a~uniform upper bound for the probability of the
former option is $\alpha(N-k)/N$. Thus, if $\GG_k$ is the
$\sigma$-field generated by $\h_{[0, S_k]}$, then
%
%
\begin{eqnarray}
\label{eq:Ek}
{\mathbb P}_{\h}
(
\EE_k ; \NN_k =\ell_k
|\GG_k)&\leq&
\frac{\alpha}{N}
\biggl(
\frac{\alpha(N-k)}{N}
\biggr)^{\ell_k -1}
\Biggl(\sum_{j=0}^\infty\biggl(\frac{k}{N}\biggr)^j\Biggr)^{\ell_k }
\nonumber\\[-8pt]\\[-8pt]
& = &\frac{\alpha^{\ell_k}}{N-k}.\nonumber
\end{eqnarray}
Therefore,
%
%
\begin{equation}\label{proof.10}
\E_{\h}\Biggl(\prod_{\ell=0}^k \mathbh{1}_{\EE_\ell} \mathbh
{1}_{\{ \NN_k = \ell_k \}}| \GG_k\Biggr)\leq
\frac{\alpha^{\ell_k }}{N-k}
\prod_{\ell=0}^{k-1}\mathbh{1}_{\EE_\ell[\ui]}
\end{equation}
and, consequently,
%
%
\begin{equation}\label{proof.11}
{\mathbb P}_{\h}(\EE[\ui] ; \NN_0 =\ell_0;\ldots;\NN_{N-1}
=\ell_{N-1} )\leq\frac{1}{N!}
\alpha^{\sum\ell_k }.
\end{equation}
As a~result, we get that
%
%
\begin{equation}
\label{eq:NNBound}
{\mathbb P}_{\h}(\NN> M)\leq\sum_{L> M}\alpha^L
\pmatrix{{N+L}\cr L}
\leq\sum_{L> M} e^{-\ln(1/\alpha)L}
e^{N(\ln(1+L/N)+1)}.\hspace*{-35pt}
\end{equation}
For $M\equiv c_2 N$ and providing that $c_2$ is large enough,
we finally obtain that
%
%
\begin{equation}
\label{eq:VirginFlips1}
{\mathbb P}_{\h}(\NN> M)\leq e^{-c_5 N}
\end{equation}
for a~constant, $c_5$, increasing linearly with $c_2$.
Putting all estimates together concludes the proof of the lemma.
\end{pf}

Notice that if ${\mathcal A}\cap\BB$ happens, then
$\un m (\h_t )\equiv\un m (\s_t )$. In
particular, ${\mathcal A}\cap\BB\subset\{\tau_B^\sigma\geq
N^\kappa\}$ and hence
$\tau_B^\sigma= \tau_B^\h$ on ${\mathcal A}\cap\BB$.

Let us go back to the cycle decomposition (\ref{eq:tauBDecomp}).
Using $\wt\E$ for the expectation on the enlarged probability space,
%
%
\begin{equation}\label{tt.5}
\E\tau_B^\s\geq\sum_{k=0}^\infty
\wt\E
\Biggl\{\prod_{\ell=0}^{k-1}\mathbh{1}_{\DD^{\ell}}(1- \mathbh
{1}_{{\mathcal A}^{\ell}}\mathbh{1}_{\BB^{\ell}})\Biggr\}
\tau_B^{k,\h} \mathbh{1}_{{\mathcal A}^{k}}\mathbh{1}_{\BB^{k}}.
\end{equation}

Let $\FF_{\theta_{k}}$ be the $\sigma$-algebra generated by all the
events and trajectories~${\mathcal A}^{\ell},\break \BB^{\ell}, \DD^{\ell}$,
$\eta^\ell$ and $\sigma(\theta_{\ell- 1}, \theta_\ell]$, $\ell\leq k$.
In view of the independence of the copies $\{\eta^\ell, \un V^\ell \}$,
%
%
\begin{equation}\label{proof.14}\qquad
\wt\E(
\tau_B^{\h} \mathbh{1}_{{\mathcal A}^{k}}\mathbh{1}_{\BB^{k}}
| \FF_{\theta_{k-1}} )
={\mathbb P}({\mathcal A})\E_{\h}(\tau_B\mathbh{1}_\BB)=
(1-\nu)^{M}\E_{\h}(\tau_B\mathbh{1}_\BB).
\end{equation}
On the other hand,
%
%
\begin{eqnarray}\label{anton.1001}
\wt\E\bigl(\mathbh{1}_{\DD^{\ell}}(1- \mathbh{1}_{{\mathcal
A}^{\ell}}\mathbh{1}_{\BB^{\ell}})|
\FF_{\theta_{\ell-1} }\bigr)
&\geq&
\wt\E(1- \mathbh{1}_{{\mathcal A}^{\ell}}\mathbh{1}_{\BB^{\ell
}})-\max_{\s'\in{\mathcal S}^n[\un m
]}\bigl(1 -{\mathbb P}_{\s'} (\DD)\bigr)\nonumber\hspace*{-35pt}\\[-8pt]\\[-8pt]
&\geq& 1 - (1-\nu)^{M} - e^{-cN}.\nonumber\hspace*{-35pt}
\end{eqnarray}
Altogether (recall that $M = c_2 N$),
%
%
\begin{eqnarray}\label{ale.1}\quad
\E_\s\tau_B
&\geq&
\E_{\h} (\tau_B \mathbh{1}_\BB)
(1-\nu)^{c_2N }
\sum_{k=0}^\infty
\bigl(1 - (1-\nu)^{c_2N} - e^{-c N}\bigr)^k
\nonumber\\[-8pt]\\[-8pt]
&\geq&\E_{\h} \tau_B \frac{1-e^{-cN}}{1+(1-\nu)^{-c_2N}
e^{-c N}},\nonumber
\end{eqnarray}
which tends to $\E_{\h} \tau_B$ if $\nu< c /c_2$.
This concludes the proof of Theorem~\ref{pointwise.1}.

\subsection{\texorpdfstring{Extension to the case $\un m(\s)\neq \un m (\eta)$}
{Extension to the case m(sigma) /= m(eta)}}

Very little has to be changed if we replace the condition that we start in
a configuration $\s$ that has the same mesoscopic magnetization
as $\eta$, but for which (\ref{intro.10}) still holds.
In that case, we cannot start the coupling in the first cycle,
so we simply have to
wait until time $\Delta_{0}$ (provided $\cD^{0}$ occurs, i.e., $\s
_t$ does
not hit
$B$ before that time). This means that we replace (\ref
{eq:tauBDecomp}) by
%
%
\begin{eqnarray}\qquad\quad
\label{eq:tauBDecomp-bis}
\tau_B^\s&=& \sum_{k=1}^\infty
\Biggl\{\mathbh{1}_{\cD^{0}}\prod_{\ell=1}^{k-1}\mathbh{1}_{\DD
^{(\ell)}}
(1- \mathbh{1}_{{\mathcal A}^{\ell}}\mathbh{1}_{\BB^{\ell}})
\Biggr\}
(\theta_{k-1} +
\tau_B^{k,\h} )\mathbh{1}_{{\mathcal A}^{k}}\mathbh{1}_{\BB^{k}}
+\tau_{B}^\s\mathbh{1}_{(\cD^{0})^c}\nonumber\\[-8pt]\\[-8pt]
&&{}+\tau_B^\s\sum_{k=1}^\infty
\Biggl\{\mathbh{1}_{\cD^{0}} \prod_{\ell=1}^{k-1}\mathbh{1}_{\DD
^{\ell}}
(1- \mathbh{1}_{{\mathcal A}^{\ell}}\mathbh{1}_{\BB^{\ell}})
\Biggr\}
(1- \mathbh{1}_{\DD^{k}})(1- \mathbh{1}_{{\mathcal
A}^{k}}\mathbh{1}_{\BB^{k}}).
\nonumber
\end{eqnarray}
We then proceed exactly as before to get
%
%
\begin{eqnarray}\label{ale.1-1}\qquad\quad
\E_{\s} \tau_B
&\geq& \E_\eta(\tau_B\mathbh{1}_{\cB})(1-\nu)^{c_2N }
\sum_{k=0}^\infty
(1-e^{-c_2N}) \bigl(1 - (1-\nu)^{c_1N} - e^{-c_2
N}\bigr)^{k}\nonumber\\[-8pt]\\[-8pt]
&\geq&\E_{\h} \tau_B \frac{(1-e^{-cN})(1-e^{-c_2N})}
{1+(1-\nu)^{-c_1N}e^{-c_2N}},\nonumber
\end{eqnarray}
which is virtually equivalent to the previous case.

\subsection{The Laplace transform}

Next, we show that the same coupling can also be used to show that the
Laplace transform of $\tau_B$ depends very little on the
initial conditions within a~set $A$.
Set $T\equiv\E_{\nu_{A}} \tau_B $.
\begin{proposition}\label{prop:pointlaplace}
If $A,B$ satisfy the hypothesis of Theorem~\ref{pointwise.1}, then,
for every configurations $\s, \h\in A$ and $\lambda\geq0$,
%
%
\begin{equation}\label{eq:pointlaplace}
R_\s(\lambda)\equiv\E_\s\bigl(e^{-({\lambda}/T)\tau
_B}\bigr)=
\E_\h\bigl(e^{-({\lambda}/T)\tau_B}\bigr)\bigl(1+ o(1)\bigr).
\end{equation}
\end{proposition}

The proof of Proposition~\ref{prop:pointlaplace} involves some
estimates and computations that we collect in the following lemmas.
\begin{lemma}\label{lemma:eventB}
There exists a~constant, $c>0$, independent of $n$, such that,
for any $\h\in A$,
%
%
\begin{equation}\label{eq:eventBexp}
\E_\h\bigl(\mathbh{1}_{\BB}e^{-({\lambda}/{T})\tau
_B}\bigr)\geq
\E_\h\bigl(e^{-({\lambda}/{T})\tau_B}\bigr)(1-e^{-c N}).
\end{equation}
\end{lemma}
\begin{pf}
The proof is similar to that of (\ref{eq:BExp}) and uses
some of the estimates given there.
The aim is to prove that
%
%
\begin{equation}\label{eq:ComplB}
\E_\h\bigl(\mathbh{1}_{\BB^c}e^{-({\lambda}/{T})\tau
_B}\bigr)\leq
\E_\h\bigl(e^{-({\lambda}/{T})\tau_B}\bigr)e^{-c N}.
\end{equation}
By Jensen's inequality, for every $\h\in A$,
%
%
\begin{equation}\label{eq:trivialb}
\E_\h\bigl(e^{-({\lambda}/T)\tau_B}\bigr)
\geq
e^{-({\lambda}/T) \E_\h\tau_B}
=
e^{-\lambda}\bigl(1+o(1)\bigr),\vadjust{\goodbreak}
\end{equation}
where the second line follows form the pointwise estimate on
$\E_\h\tau_B$ that was proven in the previous subsections.
To prove (\ref{eq:ComplB}), it is enough
to notice that, by Lemma~\ref{event.8},
%
%
\begin{equation}\label{tra.6}
\E_\h\bigl(\mathbh{1}_{\BB^c}e^{-({\lambda}/{T})\tau
_B}\bigr) \leq e^{-c N}.
\end{equation}
\upqed\end{pf}
\begin{pf*}{Proof of Proposition~\ref{prop:pointlaplace}}
For simplicity, we consider the case when $\un m (\s)=\un m (\h
)\equiv\un m$.
Analogously to (\ref{eq:tauBDecomp}), we obtain
%
%
\begin{eqnarray}\label{eq:expTB}\quad
&&\E_\s\bigl(e^{-({\lambda}/T)\tau_B} \bigr)\nonumber\\[-2pt]
&&\qquad=\wt\E\Biggl(\sum_{k=0}^\infty
e^{-({\lambda}/T)(\theta_{k-1} +\tau_B^{k,\h})}
\mathbh{1}_{{\mathcal A}^{k}}\mathbh{1}_{\BB^{k}}\prod_{\ell
=0}^{k-1}\mathbh{1}_{\DD^{\ell}}
(1- \mathbh{1}_{{\mathcal A}^{\ell}}\mathbh{1}_{\BB^{\ell}})
\Biggr)\nonumber\\[-2pt]
&&\qquad\quad{}+
\wt\E\Biggl(e^{-({\lambda}/T)\tau_B^\s} \sum
_{k=0}^\infty
(1- \mathbh{1}_{\DD^{k}})(1-
\mathbh{1}_{{\mathcal A}^{k}}\mathbh{1}_{\BB^{k}})
\prod_{\ell=0}^{k-1}\mathbh{1}_{\DD^{\ell}}(1- \mathbh
{1}_{{\mathcal A}^{\ell}}
\mathbh{1}_{\BB^{\ell}})\Biggr)
\\[-2pt]
&&\qquad\leq
\sum_{k=0}^\infty
\wt\E\Biggl(e^{-({\lambda}/T)\tau_B^{k,\h}}
\mathbh{1}_{{\mathcal A}^{k}}\mathbh{1}_{\BB^{k}}
\prod_{\ell=0}^{k-1}\mathbh{1}_{\DD^{\ell}}(1- \mathbh
{1}_{{\mathcal A}^{\ell}}\mathbh{1}_{\BB^{\ell}})
\Biggr)\nonumber\\[-2pt]
&&\qquad\quad{}+
\sum_{k=0}^\infty
\wt\E\Biggl( (1- \mathbh{1}_{\DD^{k}})(1-
\mathbh{1}_{{\mathcal A}^{k}}\mathbh{1}_{\BB^{k}})
\prod_{\ell=0}^{k-1}\mathbh{1}_{\DD^{\ell}}(1- \mathbh
{1}_{{\mathcal A}^{\ell}}
\mathbh{1}_{\BB^{\ell}})\Biggr).
\nonumber
\end{eqnarray}

Now, for every $k, \ell\geq0$, as in (\ref{proof.14}),
%
%
\begin{equation}\label{eq:expDecompI}
\wt\E\bigl(\mathbh{1}_{{\mathcal A}^{k}}\mathbh{1}_{\BB^{k}}
e^{-({\lambda}/{T})\tau_B^{k,\h}}
|\FF_{\theta_{k-1}}\bigr) \leq(1-\nu)^M \E_\h\bigl(e^{-({\lambda}/{T})\tau_B}\bigr).
\end{equation}

Moreover, as in (\ref{anton.1001}),
%
%
\begin{eqnarray}\label{anton.1002}
\wt\E\bigl(\mathbh{1}_{\DD^{(\ell)}}(1- \mathbh{1}_{{\mathcal
A}^{(\ell)}}\mathbh{1}_{\BB^{\ell}})|
\FF_{\theta_{\ell-1} }\bigr)
&\leq&
\wt\E(1- \mathbh{1}_{{\mathcal A}^{(\ell)}}\mathbh{1}_{\BB
^{\ell}}|
\FF_{\theta_{\ell-1} })\nonumber\\[-2pt]
&= & 1 - {\mathbb P}({\mathcal A}){\mathbb P}(\BB)\\[-2pt]
&\leq& 1- (1-\nu)^{M}(1 -e^{-c N}).\nonumber
\end{eqnarray}

This last estimate, together with (\ref{event.11})
of Lemma~\ref{event.8}, shows that the term in the last line of (\ref
{eq:expTB})
is smaller than
%
%
\begin{equation}\label{eq:boundexpTB}
\sum_{k=0}^\infty e^{-c N}
\bigl(1- (1-\nu)^{M}(1 - e^{-c N})\bigr)^{k}
\leq2 e^{-N(c-c_2\nu)}.
\end{equation}
Combining these estimates, we arrive at
%
%
\begin{eqnarray}\label{eq:expDecom2}\qquad
\E_\s\bigl(e^{-({\lambda}/T)\tau_B}\bigr)&\leq&
\E_\h\bigl(e^{-({\lambda}/T)\tau_B}\bigr)(1-\nu)^{c_2
N} \sum_{k=0}^\infty
\bigl(1-(1-\nu)^{c_2 N}(1-e^{-c N})\bigr)^k
\nonumber\\[-2pt]
&\leq&
\E_\h\bigl(e^{-({\lambda}/T)\tau_B}\bigr)
(1-e^{-c N})+ 2e^{-N(c-c_2\nu)}
\\[-2pt]
&=& \E_\h\bigl(e^{-({\lambda}/T)\tau_B}\bigr)
\bigl(1 + 3e^{-N(c-c_2\nu)}\bigr),\nonumber
\end{eqnarray}
which tends to $\E_\h(e^{-({\lambda}/T)\tau
_B})$ if $\nu< c/c_2$.\vadjust{\goodbreak}
\end{pf*}

\section{Renewal and the exponential distribution for the RFCW}\label{Sect.4}

We will use the results of Section~\ref{Sect.2} and the notation
introduced therein. In particular, for each $n$ fixed, we set
$A = \cS^n [\un m^*]$ and
$A_\delta= \cS^n [\mathbf A_\delta]$, where $\mathbf A_\delta$ is
the mesoscopic
$\delta$-neighborhood of $\un m^*$. In the sequel we choose $n$ appropriately
large and $\delta$ appropriately small.

In the case of the RFCW model, we prove the convergence of the law
of the normalized metastable time, $\tau_B$, to an exponential
distribution, via convergence of the Laplace transform, $R_\s(\lambda
)$, defined in
(\ref{eq:pointlaplace}).
The proof of the latter is based on renewal arguments.

\subsection{Renewal equations}
By Proposition~\ref{prop:pointlaplace}, instead of studying the
process starting in a~given point, $\s$, for which no exact renewal equation
will hold, it is enough to study the process starting on a~suitable measure
on $A$, for which such a~relation will be shown to hold.
For $\lambda\geq0$, let $\rho_\lambda$ denote the probability
measure on $A$ that
satisfies the equation
%
\begin{equation}\label{renewal.2.bis}
\sum_{\s\in A}\rho_\lambda(\s) \E_\s\bigl(e^{-({\lambda}/T)\tau_A}
\mathbh{1}_{\tau_A<\tau_B}\mathbh{1}_{\s(\tau_A)=\s'}
\bigr)=C(\lambda)\rho_\lambda(\s')
\end{equation}
for all $\s'\in A$, where
%
%
\begin{equation}\label{renewal.3}
C(\lambda)=\E_{\rho_\lambda} \bigl(e^{-({\lambda}/T)\tau_A}
\mathbh{1}_{\tau_A<\tau_B}\bigr).
\end{equation}
Existence and uniqueness of such a~measure follow in a~standard way
from the Perron--Frobenius theorem.

The usefulness of this definition comes from the fact that
the Laplace transform of $\tau_B$ started in this measure satisfies an
exact renewal equation.
\begin{lemma}\label{renewal.4}
Let $R_{\rho_\lambda}(\lambda)= \sum_\s\rho_\lambda(\s)R_\s
(\lambda)$. Then
%
%
\begin{equation}\label{renewal.5}
R_{\rho_\lambda}(\lambda)=\frac{\E_{\rho_\lambda} (
e^{-({\lambda}/T)\tau_B}
\mathbh{1}_{\tau_B<\tau_A})}
{1-\E_{\rho_\lambda}(
e^{-({\lambda}/T)\tau_A}
\mathbh{1}_{\tau_A<\tau_B})}.
\end{equation}
\end{lemma}
\begin{pf}
Using that $1=\mathbh{1}_{\tau_B<\tau_A}+\mathbh{1}_{\tau_A<\tau
_B}$ and the strong Markov property,
we see that
%
%
\begin{eqnarray}\label{renewal.6}
R_{\rho_\lambda}(\lambda)
&=&\E_{\rho_\lambda} \bigl(e^{-({\lambda}/T)\tau_B}
\mathbh{1}_{\tau_B<\tau_A}\bigr)
\nonumber\\
&&{}
+\sum_{\s'\in A}\E_{\rho_\lambda} \bigl(e^{-({\lambda}/T)\tau_A}
\mathbh{1}_{\tau_A<\tau_B}\mathbh{1}_{\s(\tau_A)=\s'}\bigr)\E
_{\s'}e^{-({\lambda}/T)\tau_B}
\\
&=&
\E_{\rho_\lambda} \bigl(e^{-({\lambda}/T)\tau_B}
\mathbh{1}_{\tau_B<\tau_A}\bigr)
+\sum_{\s'\in A} C(\lambda)\rho_\lambda(\s')\E_{\s'}
e^{-({\lambda}/T)\tau_B}.\nonumber
\end{eqnarray}
Equation (\ref{renewal.5}) is now immediate.
\end{pf}

\subsection{Convergence}
As a~result of the representation (\ref{renewal.5}), Theorem \ref
{exponential.1} will follow from
(\ref{eq:pointlaplace}) once we prove the following lemma.\vadjust{\goodbreak}
\begin{lemma}\label{steplemma.1} With the notation from Lemma \ref
{renewal.4}, for any $\lambda\geq0$,
%
%
\begin{equation}
\label{eq:explimit}
\lim_{N\to\infty}
\frac{\E_{\rho_\lambda} (e^{-({\lambda}/T)\tau_B}
\mathbh{1}_{\tau_B<\tau_A})}
{1-\E_{\rho_\lambda}(
e^{-({\lambda}/T)\tau_A}
\mathbh{1}_{\tau_A<\tau_B})} = \frac{1}{1+\lambda} .
\end{equation}
\end{lemma}
\begin{pf}
The proof of this lemma comprises seven steps.

\textsc{Step} 1. Define $T_\lambda= \E_{\rho_\lambda}$. We claim:
\begin{lemma}\label{renewal.13}
There exists $c_6>0$, such that, for any $\lambda\geq0$ fixed,
%
%
\begin{equation}
\label{eq:TlambdaFrac}
T_\lambda= \frac{\E_{\rho_\lambda}\tau_{A\cup B }}
{{\mathbb P}_{\rho_\lambda}(\tau_B <\tau_A)}
\bigl(1 + o (e^{-c_6 N})\bigr).
\end{equation}
\end{lemma}

Indeed,
%
%
\begin{eqnarray}\label{Tlambda}\qquad
T_\lambda&=&
\E_{\rho_\lambda}\bigl(\tau_B\mathbh{1}_{\{\tau_B <\tau
_A\}}\bigr) +
\E_{\rho_\lambda}\bigl(\tau_B\mathbh{1}_{\{\tau_A<\tau_B\}
}\bigr)\nonumber\\
&=&\E_{\rho_\lambda} \tau_{A\cup B}
+
\E_{\rho_\lambda}\bigl(\mathbh{1}_{\{\tau_A <\tau_B\}}\E
_{\s(\tau_A )}\tau_B\bigr)\\
&=& \E_{\rho_\lambda}\tau_{A\cup B} + T_\lambda{\mathbb P}_{\rho
_\lambda}(\tau_A <\tau_B)
+\E_{\rho_\lambda}\bigl(\mathbh{1}_{\{\tau_A <\tau_B\}}\bigl(
\E_{\s(\tau_A )}\tau_B -T_\lambda\bigr)\bigr).\nonumber
\end{eqnarray}
However, by the invariance of $\rho_\lambda$,
%
%
\begin{equation}\label{tt.6}
\E_{\rho_\lambda}\bigl(\mathbh{1}_{\{\tau_A <\tau_B\}}
e^{-{({\lambda}/T)}\tau_A}
\bigl(\E_{\s(\tau_A )}\tau_B -T_\lambda\bigr)\bigr)
= 0.
\end{equation}
It follows that the absolute value of the last term in (\ref{Tlambda})
is bounded above
as
%
%
\begin{eqnarray}
&&\E_{\rho_\lambda}\bigl(\bigl(1 - e^{ -{({\lambda}/{T})}\tau
_A} \bigr)
\mathbh{1}_{\{\tau_A<\tau_B\}}\bigr)
\max_{\s\in A} | \E_\s\tau_B - T_\lambda|\nonumber\\
&&\qquad\leq\lambda
\max_{\s\in A} \biggl| \frac{\E_\s\tau_B - T_\lambda}{T}\biggr|
\E_{\rho_\lambda}\bigl(\tau_A\mathbh{1}_{\{\tau_A<\tau_B\}
} \bigr)\\
&&\qquad=
o (e^{-C N/2} ) \E_{\rho_\lambda}\bigl(\tau_A\mathbh
{1}_{\{\tau_A<\tau_B\}}\bigr) ,
\nonumber
\end{eqnarray}
where we used (\ref{intro.11}) in the last step. This implies the
claim of the
lemma.

\textsc{Step} 2. Control of $\rho_\lambda$-measure.
\begin{lemma}\label{close-measures.1}
There exists $c_7 <\infty$, such that for any
$n$ [and hence $\epsilon= \epsilon(n)$] fixed,
%
%
\begin{equation}
\label{eq:rholmu}
\max_{\s\in A}\frac{\rho_\lambda(\s)}{\mu(\s)/\mu(A )}\leq
e^{c_7\epsilon N}
\end{equation}
as soon as $N$ is large enough.
\end{lemma}
\begin{pf}
In order to prove (\ref{eq:rholmu}), first of all, note that by
reversibility
%
%
\begin{equation}
\label{eq:muAinv}
\sum_{\s'\in A}\mu(\s' ) {\mathbb P}_{\s'}\bigl(\tau_A^r <\tau_B ;
\s(\tau_A^r ) = \s\bigr)
= \mu(\s){\mathbb P}_\s(\tau_A^r <\tau_B),
\end{equation}
where $\tau_A^r$ is the $r$th hitting time of $A$. Assume now that we
are able to prove that there exists $r$ and $M$ such that
%
%
\begin{equation}
\label{eq:threepoints}\qquad
{\mathbb P}_\h\bigl(\tau_A^r <\tau_B ; \s(\tau_A^r ) = \s\bigr)\leq
(1- \epsilon)^{-M}
{\mathbb P}_{\s'}\bigl(\tau_A^r <\tau_B ; \s(\tau_A^r ) = \s\bigr),
\end{equation}
uniformly in $\h,\s,\s'\in A$. In view of (\ref{renewal.2.bis}),
this would
imply
%
%
\begin{eqnarray}
\rho_\lambda(\s) & \leq &\frac{1}{C(\lambda)^r} \sum_{\h} \rho
_\lambda(\h)
{\mathbb P}_\h\bigl(\tau_A^r <\tau_B ; \s(\tau_A^r ) =
\s\bigr)\nonumber\\[-8pt]\\[-8pt]
& \leq &
\frac{(1- \epsilon)^{-M}}{C (\lambda)^r}
{\mathbb P}_{\s'} \bigl(\tau_A^r <\tau_B ; \s(\tau_A^r ) = \s\bigr).
\nonumber
\end{eqnarray}
Multiplying both sides above by $\mu(\s' )$ and applying
(\ref{eq:muAinv}), we conclude that~(\ref{eq:threepoints}) implies
that
%
%
\begin{equation}
\label{eq:MrBound}
\rho(\s)\leq\frac{(1- \epsilon)^{-M}}{C (\lambda)^r}\frac
{\mu(\s)}{\mu(A )}
{\mathbb P}_{\s} (\tau_A^r <\tau_B),
\end{equation}
uniformly in $\s\in A$. The target (\ref{eq:rholmu}), therefore, will
be a~consequence of the following two claims: there exists
$c >0 $, such that, independently of the coarse graining parameter $n$,
%
%
\begin{equation}
\label{eq:ClBound}
C(\lambda)\geq1 - e^{-c N}
\end{equation}
as soon as $N$ is sufficiently large. Furthermore, for sufficiently
large $c_2$ and~$\kappa$, (\ref{eq:threepoints}) holds with
$M= c_2 N$ and $r = N^\kappa$.

We first show that (\ref{eq:ClBound}) holds. By the uniform bound
(\ref{hypothesis.2}) and Jensen's inequality, it follows that
%
%
\begin{eqnarray}\label{tt.40}
C(\lambda) &\geq& (1 - e^{-cN})\sum_\s\rho_\lambda(\s)
\E_\s\bigl(e^{-({\lambda}/T)\tau_A }| \tau_A <\tau
_B\bigr)\nonumber\\
&\geq&
(1 - e^{-cN})
\exp\biggl\{- {\frac{\lambda}{T}}\sum_\s\rho_\lambda(\s)
\E_\s(\tau_A | \tau_A <\tau_B)\biggr\}\\
&\geq&
(1 - e^{-cN})
\exp
\biggl\{- \frac{\lambda\E_{\rho_\lambda}( \tau_A\mathbh
{1}_{\{\tau_A <\tau_B\}})}{T
(1 - e^{-cN})
}\biggr\}.
\nonumber
\end{eqnarray}
By (\ref{eq:TlambdaFrac}),
%
%
\begin{equation}\label{tt.15}
\E_{\rho_\lambda} \bigl(\tau_A\mathbh{1}_{\{\tau_A <\tau
_B\}}\bigr) \leq
\frac{T_\lambda{\mathbb P}_{\rho_\lambda}(\tau_B <\tau_A)}
{1 + {o}(e^{-c_8 N})} ,
\end{equation}
and (\ref{eq:ClBound}) follows by (\ref{hypothesis.2}) and (\ref{intro.11}).

Next, we show that (\ref{eq:threepoints}) holds.
There exists
$c_{8} <\infty$ such that
%
%
\begin{equation}
\label{eq:ssprimeBound}
{\mathbb P}_{\h}(\tau_A^r <\tau_B ; \s_{\tau_A^r}= \s)
\geq e^{-c_{8}N}
{\mathbb P}_{\h}(\tau_A^r <\tau_B),
\end{equation}
uniformly in $\s,\h\in A$.
This is a~rough estimate: by the Markov property,
%
%
\begin{equation}\label{tt.16}\qquad
{\mathbb P}_{\h}(\tau_A^r <\tau_B ; \s_{\tau_A^r}= \s)
\geq{\mathbb P}_{\h}(\tau_A^r <\tau_B)
\min_{\h'\in A} {\mathbb P}_{\h'}(\tau_A <\tau_B ; \s
_{\tau_A }= \s).
\end{equation}
Let $\h'\in A$ and let the Hamming distance between $\s$ and $\h'$
be $K$. Then we can reach $\s$ from $\h'$ by flipping exactly $K$
spins; since this
can be done in~$K!$ orders, and each flip has probability at least
$((1-\alpha)/N)$ by
(\ref{eq:PUpdate}), we see that
%
%
\begin{equation}\label{tt.17}
{\mathbb P}_{\h'}(\tau_A <\tau_B ; \s_{\tau_A }= \s)
\geq\frac{K!}{N^{K}}
(1-\alpha)^{K},
\end{equation}
and (\ref{eq:ssprimeBound}) follows.

Next, let $\h\in A$ and consider a~dynamics starting from $\h$. We
shall try
to couple it with a~dynamics starting from $\s'$ using just one \textit{basic
coupling attempt}. Employing the same notation as in
Section~\ref{section.2.2}, we know [see (\ref{proof.5}) and
(\ref{eq:VirginFlips1})] that for $\kappa>2$ and $M= c_2N$,
%
%
\begin{equation}
\label{eq:c6}
{\mathbb P}_\h(\cS>N^\kappa, \cN> M)\leq e^{- c_9 N},
\end{equation}
where $c_9$ grows linearly with $c_2$. In the sequel, we choose $c_2$
so large that $c_9 $ becomes larger than the constant $c_{8}$ in (\ref
{eq:ssprimeBound}).

Let us redefine the event $\cB$ in (\ref{event.5}) as
$\cB= \{\cS\leq N^\kappa\}\cap
\{\cN\leq M\}$. The coins $V_1,\ldots,V_M$ and the event
$\cA=\{\bigforalll V_i = 1\}$ remain the same.
Consider the enlarged probability
space $(\wt\Omega,\wt{\mathbb P})$ which corresponds to
a single basic coupling
attempt to couple a~dynamics $\s(t)$ from $\s'$ to the dynamics
$\eta(t)$ which starts at~$\eta$.

The coupling is successful if and only if the event $\cA\cap\cB$, which
depends on at most $N^\kappa$ steps, happens. Therefore,
%
%
\begin{eqnarray}\label{trala.1}
&&{\mathbb P}_{\s'}\bigl(\tau_{A}^r <\tau_B ; \s(\tau_{A}^r )=\s
\bigr)\nonumber\\
&&\qquad\geq
\wt{\mathbb P} \bigl( N^\kappa\leq\tau_{A}^r <\tau_B ; \h(\tau
_{A}^r ) =\s;\cA; \cB\bigr)\nonumber\\[-8pt]\\[-8pt]
&&\qquad={\mathbb P}_\h\bigl( N^\kappa\leq\tau_{A}^r
<\tau_B ; \h(\tau_{A}^r ) =\s;\cB\bigr)\wt{\mathbb P} (\cA
)\nonumber\\
&&\qquad= {\mathbb P}_\h\bigl( N^\kappa\leq\tau_{A}^r
<\tau_B ; \h(\tau_{A}^r ) =\s;\cB\bigr)(1-\epsilon)^M.\nonumber
\end{eqnarray}
Now, let us choose $r = N^\kappa$. In particular, the constraint
$N^\kappa\leq\tau_A^r$ becomes redundant. By (\ref{hypothesis.2})
and in view
of (\ref{eq:ssprimeBound}) and our choice of $M$ which
leads to a~large $c_{9}$ in (\ref{eq:c6}), there exists $c_{10}>0$
such that
%
%
\begin{eqnarray}\label{trala.2}
&&{\mathbb P}_\h\bigl(\tau_{A}^r <\tau_B ; \h(\tau_{A}^r ) =\s;\cB
\bigr)\nonumber\\
&&\qquad\geq
{\mathbb P}_\h\bigl(\tau_{A}^r <\tau_B ; \h(\tau_{A}^r ) = \s\bigr)
-{\mathbb P}_\h(\cB^c )
\\
&&\qquad
\geq {\mathbb P}_\h\bigl(\tau_{A}^r <\tau_B ; \h(\tau_{A}^r ) =\s
\bigr)(1 - e^{-c_{10}N}).\nonumber
\end{eqnarray}
Equation (\ref{eq:threepoints}) follows.
\end{pf}

\textsc{Step} 3.
The following crucial bound, to which we refer to a~\textit{uphill lemma},
will be proven
in the next subsection.
\begin{lemma}\label{uphill.1} There exists $c_{11}>0$ such that
%
%
\begin{equation}
\label{eq:rholbound}
\E_{\rho_\lambda}(\tau_B \mathbh{1}_{\tau_B <\tau_A})
\leq e^{-c_{11}N} \E_{\rho_\lambda}(\tau_A \mathbh
{1}_{\tau_A<\tau_B}) .
\end{equation}
%
\end{lemma}
\begin{remark*}
Intuitively, the bound (\ref{eq:rholbound}) should follow from the
decomposition
%
%
\begin{eqnarray}\label{intuition.1}
\E_\s\tau_{A\cup B} &=& {\mathbb P}_\s(\tau_A<\tau_B)
\E_\s(\tau_A|\tau_A<\tau_B)\nonumber\\[-8pt]\\[-8pt]
&&{}+{\mathbb P}_\s(\tau_B<\tau_A) \E_\s(\tau
_B|\tau_B<\tau_A)\nonumber
\end{eqnarray}
since the first probability on the right-hand side is close to one, the second
is exponentially small, and the two conditional expectations should be
of the
same order. It seems, however, remarkably difficult to establish such a~result
uniformly in the starting point $\s\in A$, for the same reasons why the
pointwise control of mean exit times is difficult.
\end{remark*}

We shall proceed with the proof assuming that (\ref{eq:rholbound}) holds.

\textsc{Step} 4.
In view of (\ref{eq:TlambdaFrac}), a~look at (\ref{eq:rholbound})
reveals that
the conditional expectation
%
%
\begin{equation}\label{tt.7}
\frac{\E_{\rho_\lambda}( \tau_B|\tau_B <\tau_A
)}{T} = o (e^{-c_{11} N}) .
\end{equation}
Using that, for $x\geq0$, $1\geq e^{-x}\geq1- x$, it follows
that the
numerator in (\ref{renewal.5}) satisfies
%
%
\begin{equation}\label{doch.1}
\E_{\rho_\lambda} \bigl(e^{-({\lambda}/T)\tau_B}
\mathbh{1}_{\tau_B<\tau_A}\bigr)=
{\mathbb P}_{\rho_\lambda}(\tau_B <\tau_A)\bigl(1+o(e^{-c_{11}N})\bigr).
\end{equation}

\textsc{Step} 5.
Let us turn now to the denominator in (\ref{renewal.5}). We rewrite it as
%
%
\begin{equation}\label{tt.8}
{\mathbb P}_{\rho_\lambda}(\tau_B <\tau_A)\biggl(1 +
\lambda
\frac{\E_{\rho_\lambda}((1 - e^{ -{({\lambda
}/{T})}\tau_A})
\mathbh{1}_{\{\tau_A <\tau_B\}})}{\lambda{\mathbb
P}_{\rho_\lambda}(\tau_B <\tau_A)}\biggr).
\end{equation}
Using (\ref{eq:TlambdaFrac}) for $1/{\mathbb P}_{\rho_\lambda}(
\tau_B <\tau_A)$, we are
left with the computation of
%
%
\begin{equation}
\label{eq:limtocompute}
\frac{T}{\lambda\E_{\rho_\lambda}(\tau_A \mathbh{1}_{\{
\tau_A <\tau_B\}})}
\E_{\rho_\lambda}\bigl( \bigl(1 - e^{ -{({\lambda}/{T})}\tau
_A} \bigr)
\mathbh{1}_{\{\tau_A <\tau_B\}}\bigr).
\end{equation}
Since,
%
%
\begin{equation}\label{tt.9}
\E_{\rho_\lambda} \bigl(\bigl(1 - e^{ -{({\lambda
}/{T})}\tau
_A} \bigr)
\mathbh{1}_{\{\tau_A <\tau_B\}}\bigr) =
\frac{\lambda}{T}\int_0^1\E_{\rho_\lambda}\bigl(e^{-{
({s\lambda}/{T})}\tau_A}
\tau_A \mathbh{1}_{\{\tau_A <\tau_B\}}\bigr)\,\mathtt{d}s ,\hspace*{-35pt}
\end{equation}
we deduce that the expression in (\ref{eq:limtocompute}) belongs to
the interval
%
%
\begin{equation}\label{tt.10}
\biggl[
\frac{\E_{\rho_\lambda} (e^{-{({\lambda}/T)}\tau_A}
\tau_A
\mathbh{1}_{\{\tau_A <\tau_B\}})}{\E_{\rho_\lambda}(
\tau_A
\mathbh{1}_{\{\tau_A <\tau_B\}})} , 1 \biggr].
\end{equation}
The target (\ref{eq:explimit}) follows once we show that
%
%
\begin{equation}
\label{eq:onelimit1}
\lim_{N\to\infty}
\frac{\E_{\rho_\lambda}(e^{ -{({\lambda}/{T})}\tau_A}
\tau_A
\mathbh{1}_{\{\tau_A <\tau_B\}})}{\E_{\rho_\lambda}(
\tau_A
\mathbh{1}_{\{\tau_A <\tau_B\}})} =1 .
\end{equation}
It is clear that (\ref{eq:onelimit1}) follows
as soon as we check that there exists a~sequence
$\alpha_N\downarrow0$, such that
%
%
\begin{equation}
\label{eq:onelimit2}
\lim_{N\to\infty}
\frac{\E_{\rho_\lambda} (\tau_A
\mathbh{1}_{\{\tau_A <\tau_B\}}\mathbh{1}_{\{\tau_A
<\alpha_N T\}})}
{\E_{\rho_\lambda}(\tau_A
\mathbh{1}_{\{\tau_A <\tau_B\}})} =1.
\end{equation}
This will be our next goal.

Let $B_\delta= \cS_N\setminus A_\delta$. Our proof of
(\ref{eq:onelimit2}) is based on the following decomposition:
%
%
\begin{eqnarray}
\E_{\rho_\lambda} \bigl(\tau_A
\mathbh{1}_{\{\tau_A <\tau_B\}}\mathbh{1}_{\{\tau_A
>\alpha_N T\}}\bigr)& \leq &
\E_{\rho_\lambda} \bigl(\tau_A
\mathbh{1}_{\{\tau_A <\tau_{B_\delta}\}}\mathbh{1}_{\{
\tau_A >\alpha_N T\}}\bigr)\nonumber\\
&&{}+ \E_{\rho_\lambda} \bigl(\tau_A
\mathbh{1}_{\{\tau_{B_\delta}< \tau_A <\tau_B\}}\bigr)\\
& \equiv &{\mathrm{I}}_\delta+ {\mathrm{II}}_\delta.
\nonumber
\end{eqnarray}
The logic behind this decomposition should be transparent:
the conditional (on $\tau_A <\tau_B$) landscape should have the global
mesoscopic minima at $A$.
The term $\mathrm{I}_\delta$
is a~local one and should be small, since the dynamics cannot spend
too much time
inside a~local well $A_\delta$ without hitting $A$. On the other hand,
the term
$\mathrm{II}_\delta$ should be small because of the price paid for the
uphill run
toward $B_\delta$ before hitting $A$.
We claim that there exists $\alpha_N\downarrow0$ and $c> 0$ such that
%
%
\begin{equation}
\label{eq:IandIIbound}
\max\{\mathrm{I}_\delta, \mathrm{II}_\delta\}\leq e^{-cN}
\E_{\rho_\lambda} \bigl(\tau_A
\mathbh{1}_{\{\tau_A <\tau_{B }\}}\bigr).
\end{equation}
Evidently, (\ref{eq:onelimit2}) is a~consequence of (\ref{eq:IandIIbound}).

\textsc{Step} 6. Bound on $\mathrm{I}_\delta$. The term $\mathrm
{I}_\delta
$ is bounded
above as
%
%
\begin{equation}
\label{eq:OneWellBound}
\mathrm{I}_\delta\leq\max_{\s\in A}\E_\s\bigl(\tau_{A\cup B_\delta}
\mathbh{1}_{\{\tau_{A\cup B_\delta} > \alpha_N T\}}\bigr).
\end{equation}
The right-hand side of (\ref{eq:OneWellBound}) depends on the dynamics
in a~$\delta$-neighborhood of a~nondegenerate
local minimum $A = \cS^n[\un m^*]$. We try to formalize an intuitive
idea that
such dynamics mixes up on time scale much shorter than $T$ and cannot
afford spending $\alpha_N T$ units of time without hitting $A\cup
B_\delta$. This
is a~somewhat coarse estimate. Let us start with estimating hitting times
from equilibrium measure over mesoscopic slots:
\begin{lemma}\label{local-time.1}
Let $A_\delta$ and $B_\delta$ be as defined above. Then there exists
$c(\delta)$,
satisfying $c(\delta)\downarrow0$, as $\delta\downarrow0$, such that,
for all $\un m'\in\bA_\delta\setminus\un m^*$,
%
%
\begin{equation}\label{local-time.2}
\E_{\nu_{\un m'}} \tau_{A\cup B_\delta}\leq e^{ c(\delta) N},
\end{equation}
where $\nu_{\un m'}$ is
the probability measure on $\cS^n[\un m']$,
which we referred to in (\ref{intro.2}).
\end{lemma}
\begin{pf}
By formula (\ref{intro.2}),
we have that
%
%
\begin{eqnarray}\label{local-time.3}\qquad
\E_{\nu_{\un m'}} \tau_{A\cup B_\delta} &=&\frac
{1}{\capa(\un m',A\cup B_\delta)}
\sum_{\s\in A_\delta\setminus A}\mu_{\beta,N}(\s) h_{\cS^n[\un
m'],\cS[A\cup B_\delta]}(\s)\nonumber\\[-8pt]\\[-8pt]
&\leq&
\frac{1}{\capa(\un m',A)}
\sum_{\s\in A_\delta\setminus A}\mu_{\beta,N}(\s)=\frac{\mu
_{\beta,N}(A\setminus A_\delta)}{\capa(\un m',A)}.\nonumber\vadjust{\goodbreak}
\end{eqnarray}
Note that we used here only the crudest possible estimate on the
harmonic function $h_{\cS^n[\un m'],A\cup B_\delta}(\s)$, but the
results of
\cite{BBI08} do not give us anything much better.
It remains to bound the capacity $\capa(\un m',A)$ from below.
However, this is relatively easy using the methods explained in Section
5 of
\cite{BBI08}, to which we refer for further details. One gets that
%
%
\begin{equation}\label{local-time.4}
\capa(\un m',A)
\geq e^{-c\delta{\varepsilon}N}\mu_{\beta,N}(\un m').
\end{equation}
\upqed\end{pf}

As a~consequence we obtain the following lemma.
\begin{lemma}
\label{lem:LocalTimes}
Let $A_\delta$ and $B_\delta$ be as defined above. Then there exists
$c(\delta)$
satisfying
$c(\delta)\downarrow0$ as $\delta\downarrow0$, such that, for all
$\h\in A_\delta\setminus A$,
%
%
\begin{equation}
\label{eq:LocalTimes}
{\mathbb P}_{\h}\bigl(\tau_{A\cup B_\delta}\leq2 e^{ c(\delta)
N}\bigr)\geq
\frac{(1 - \nu(n ))^{M}}{3},
\end{equation}
where $1-\nu(n)$ is the probability (\ref{2.5}) of a~successful single
coin-flip and $M= c_2 N$ is the number of coins.
\end{lemma}
\begin{pf}
As the formulation of the lemma suggests, we use the basic coupling as described
in the preceding section: let $\un m' \in\mathbf A_\delta$ and
$\eta,\s\in\cS^n[\un m']$. Define the event $\cB$ as in (\ref
{event.5}). In fact,
since we are interested in $\tau_{A\cup B_\delta}$, the first
constraint in
(\ref{event.5}) becomes redundant and we can redefine $\cB$ simply as
%
%
\begin{equation}\label{tt.11}
\cB= \{\cS< N^\kappa\}\cap\{\cN< M\}.
\end{equation}
Then, performing our basic coupling attempt we infer that,
for any $\eta,\s\in\cS[\un m']$,
%
%
\begin{equation}\label{tt.12}
{\mathbb P}_{\h}\bigl(\tau_{A\cup B_\delta}\leq2 e^{ c(\delta)
N}\bigr)\geq
\bigl(1-\nu(n)\bigr)^{M}{\mathbb P}_{\s}
\bigl(\tau_{A\cup B_\delta}\leq e^{ c(\delta) N} ; \cB\bigr).
\end{equation}
By Lemma~\ref{local-time.1} and Chebyshev's inequality
%
%
\begin{equation}\label{tt.13}
{\mathbb P}_{\nu_{\un m'}}\bigl(\tau_{A\cup B_\delta}\leq2 e^{
c(\delta) N}\bigr)
\geq\tfrac12
\end{equation}
and, in view of the bound (\ref{event.10}),
(\ref{eq:LocalTimes}) follows.
\end{pf}

Let us go back to (\ref{eq:OneWellBound}). By
Lemma~\ref{lem:LocalTimes},
%
%
\begin{equation}\label{tt.14}
\max_{\s\in A}
{\mathbb P}_\s\bigl(\tau_{A\cup B_\delta} > k 2e^{c(\delta)N} \bigr)
\leq\biggl(1- \frac{(1 - \nu(n) )^M}{3}
\biggr)^{k} .\vadjust{\goodbreak}
\end{equation}
Therefore, as follows by a~straightforward application of the tail
formula,
%
%
\begin{equation}
\label{eq:Oned}
I_\delta\leq e^{-c_{12} N}
\end{equation}
as soon as
%
%
\begin{equation}
\label{eq:Parameters}
\alpha_N T > 3 c_{13} N e^{( c(\delta) + \nu(n)) N} .
\end{equation}
Since $T\sim e^{CN}$ with $C >0$ being, of course, independent
of our choice of $\delta$ and~$n$, it is always possible to tune the parameters
$\delta$, $n$ and $\alpha_N\downarrow0$ in such a~way that
(\ref{eq:Parameters}) holds.

\textsc{Step} 7. Bound on $\mathrm{II}_\delta$.
Note that
%
%
\begin{eqnarray}
\label{eq:IId}\qquad
\E_{\rho_\lambda} \bigl( \tau_A
\mathbh{1}_{\{\tau_{B_\delta}< \tau_A <\tau_B\}}\bigr)& = &
\E_{\rho_\lambda} \bigl( \tau_{B_\delta}
\mathbh{1}_{\{\tau_{B_\delta}< \tau_A \}}\bigr)\nonumber\\[-8pt]\\[-8pt]
&&{} +
\E_{\rho_\lambda}\bigl(
\mathbh{1}_{\{\tau_{B_\delta}< \tau_A \}}
\E_{\s(\tau_{B_\delta})}\bigl(\tau_A\mathbh{1}_{\{\tau_A <\tau
_B\}}\bigr)\bigr) .
\nonumber
\end{eqnarray}
By the \textit{Uphill lemma} [see (\ref{eq:rholbound}) above]
the first term in (\ref{eq:IId}) is negligible with respect to $\E
_{\rho_\lambda} \tau_A
\mathbh{1}_{\{\tau_A <\tau_{B_\delta}\}}$.
Therefore, the bulk of the remaining work is to find an appropriate upper
bound on the second term in (\ref{eq:IId}).

By the \textit{Downhill lemma} [see (\ref{eq:downhill}) below] we would be
in good shape if we would have the original reversible measure
$\mu$ instead of the $\rho_\lambda$ eigen-measure defined in (\ref
{renewal.2.bis}).
Namely, as it is explained in the end of Section~\ref{sub:Hills},
(\ref{eq:downhill}) implies that, independently
of $n$, there exists $c_\delta>0$ such that
%
%
\begin{equation}
\label{eq:muestimate}
\frac{1}{\mu(A)}\sum_{\s\in A}\mu(\s)
\E_{\s}\bigl(
\mathbh{1}_{\{\tau_{B_\delta}< \tau_A \}}
\E_{\s(\tau_{B_\delta})}\bigl(\tau_A\mathbh{1}_{\{\tau_A <\tau
_B\}}\bigr)\bigr)
\leq e^{-c_\delta{N}}
\end{equation}
as soon as $N$ is large enough. However, since we have already
established in~(\ref{eq:rholmu}) that $\rho_\lambda$ is, up to arbitrary small
exponential corrections,
controlled by $\mu$, it follows that the second term in (\ref
{eq:IId}) is exponentially small
and hence also negligible with respect to $\E_{\rho_\lambda} (
\tau_A
\mathbh{1}_{\{\tau_A <\tau_{B_\delta}\}})$.\vspace*{1pt}

The proof of Lemma~\ref{steplemma.1} is now complete.
\end{pf}

\subsection{Uphill and Downhill lemmas}
\label{sub:Hills}
In this subsection, we shall prove~(\ref{eq:rholbound})
and (\ref{eq:muestimate}).
\begin{pf*}{Proof of Lemma~\ref{uphill.1}}
Instead of proving (\ref{eq:rholbound}) directly, we will first show the
(more natural) estimate
%
%
\begin{equation}\label{eq:pwbound-var}
\sum_{\s\in A}\mu(\s)\E_{\s}(\tau_B \mathbh{1}_{\tau_B
<\tau_A})
\leq e^{-c N} \mu(A)
\end{equation}
for some $c>0$.
To do so, we use the fact that
%
%
\begin{equation}\label{rev.1}
\E_\s\tau_{A\cup B}=\E_\s( \tau_A\mathbh{1}_{\tau_A<\tau
_B})+
\E_\s( \tau_B\mathbh{1}_{\tau_B<\tau_A}).
\end{equation}
Define the function
%
%
\begin{equation}\label{rev.2}
w_{A,B}(\s)\equiv\cases{\E_\s( \tau_A\mathbh{1}_{\tau
_A<\tau_B}),&\quad
if $\s\notin A\cup B$,\cr
0,&\quad else,}
\end{equation}
$w_{A,B}$ solves
the Dirichlet problem
%
%
\begin{eqnarray}\label{rev.3}
L w_{A,B} (\s)&=&h_{A,B}(\s),\qquad \s\notin A\cup B,\\
w_{A,B}(\s)&=&0, \s\in A\cup B,
\end{eqnarray}
where $L\equiv1-P$.
Notice that, for $\s\in A$,
%
%
\begin{equation}\label{rev.4}
\E_\s( \tau_A\mathbh{1}_{\tau_A<\tau_B})={\mathbb
P}_\s(\tau_A<\tau_B)
-Lw_{A,B}(\s).
\end{equation}
Next, using reversibility,
%
%
\begin{equation}\label{rev.5}
\sum_{\s} \mu(\s) h_{A,B}(\s) Lw_{A,B}(\s)=\sum_\s
\mu(\s) Lh_{A,B}(\s) w_{A,B}(\s).
\end{equation}
By the properties of the functions $h_{A,B}$ and $w_{A,B}$, this
equation reduces to
%
%
\begin{equation}\label{rev.6}
-\sum_{\s\in A}\mu(\s)Lw_{A,B}(\s) =\sum_{\s\notin A\cup B}\mu
(\s)
h_{A,B}(\s)^2.
\end{equation}
Hence,
%
%
\begin{eqnarray}\label{rev.7}
\sum_{\s\in A}\mu(\s)
\E_\s( \tau_A\mathbh{1}_{\tau_A<\tau_B})&=& \sum_{\s
\in A}
\mu(\s){\mathbb P}_\s(\tau_A<\tau_B)\nonumber\\[-8pt]\\[-8pt]
&&{}+\sum_{\s\notin A\cup B}\mu(\s)
h_{A,B}(\s)^2.\nonumber
\end{eqnarray}
Using a~completely similar procedure, one shows that
%
%
\begin{equation}\label{rev.8}
\sum_{\s\in A}\mu(\s)
\E_\s( \tau_{A\cup B})= \sum_{\s\in A} \mu(\s)
+\sum_{\s\notin A\cup B}\mu(\s)
h_{A,B}(\s).
\end{equation}
Therefore, taking into account (\ref{rev.1}),
%
%
\begin{eqnarray}\label{rev.9}
\sum_{\s\in A}\mu(\s)\E_\s(\tau_B\mathbh{1}_{\tau_B<\tau
_A})
&=& \sum_{\s\in A}\mu(\s) {\mathbb P}_\s(\tau_B< \tau
_A)\nonumber\\[-8pt]\\[-8pt]
&&{} +\sum_{\s\notin A\cup B}\mu(\s) h_{A,B}(\s)h_{B,A}(\s).\nonumber
\end{eqnarray}
The first term on the right-hand side is exponentially small compared to~$\mu(A)$
by Lem\-ma~\ref{lemma:hypothesis}. The same holds true for the second term,
by the
same estimates that were used in the proof of Lemmas 6.1 and 6.2 in
\cite{BBI08}.
Thus~(\ref{eq:pwbound-var}) holds. By Lemma~\ref{close-measures.1},
it follows that for a~slightly smaller constant~$c'$,
$\E_{\rho_\lambda}(\tau_B\mathbh{1}_{\tau_B<\tau_A})
\leq e^{-c'N}$. Finally,
$\E_{\rho_\lambda}\tau_{A\cup B}\geq1$ and so
$\E_{\rho_\lambda}(\tau_A\mathbh{1}_{\tau_A<\tau_B})
\geq1-e^{-c'N}$, and we can deduce (\ref{eq:rholbound}).
This concludes the proof of the lemma.~%
\end{pf*}

The microscopic harmonic function $h(\s)\equiv{\mathbb P}(\tau_A
<\tau_B)$ gives
rise to the so-called $h$-transformed chain with transition probabilities
%
%
\begin{equation}\label{h-trasfTP}
p^h_N (\s,\s^\prime) = h (\s)^{-1}p_N (\s,\s^\prime) h(\s
^\prime).
\end{equation}
This
$h$-transformed chain lives on $\{\s\dvtx h (\s)>0\}$
and it is reversible with respect to $\mu^h \equiv h^2\mu$. The following
\textit{Downhill lemma} holds.
\begin{lemma}\label{downhill.1}
With the notation introduced before,
%
%
\begin{eqnarray}
\label{eq:downhill}
&&\sum_{\s\in A}\mu(\s)\E_\s\bigl(\mathbh{1}_{\{\tau
_{B_\delta} <\tau_A\}}
\E_{\s\tau_{B_\delta}}\bigl(\tau_A\mathbh{1}_{\{\tau_A
<\tau_B \}}\bigr)\bigr)\nonumber\\[-8pt]\\[-8pt]
&&\qquad\leq
\sum_{\s'\in A^\delta\setminus A}\mu^h (\s' ){\mathbb P}^h_{\s'
}(\tau_{B_\delta} <\tau_A).
\nonumber
\end{eqnarray}
\end{lemma}
\begin{pf}
By reversibility,
%
%
\begin{equation}\label{tt.2}\quad
\mu(\s)\E_\s\bigl(\mathbh{1}_{\{\tau_{B_\delta} <\tau_A\}
}\mathbh{1}_{\{\s(\tau_{B_\delta})
=\h\}}\bigr)
= \mu(\h)\E_\h\bigl(\mathbh{1}_{\{\tau_A <\tau_{B_\delta
}\}}\mathbh{1}_{\{\s(\tau_{A})=\s\}}\bigr).
\end{equation}
Hence,
%
%
\begin{eqnarray}\label{tt.3}
&&\sum_{\s\in A}\mu(\s)\E_\s\bigl(\mathbh{1}_{\{\tau
_{B_\delta} <\tau_A\}}
\E_{\s(\tau_{B_\delta})}\bigl(\tau_A\mathbh{1}_{\{\tau_A <\tau
_B \}}\bigr)\bigr)\nonumber\\[-8pt]\\[-8pt]
&&\qquad =
\sum_{\h\in B_\delta}\mu(\h) {\mathbb P}_\h(\tau_A <\tau
_{B_\delta})
\E_\h\bigl(\tau_A\mathbh{1}_{\{\tau_A <\tau_B\}}\bigr) .
\nonumber
\end{eqnarray}
Since the only nonzero contribution to the latter sum comes from
$\eta$ in the exterior boundary of $A_\delta$,
we can bound it from above in terms of the $h$-transformed
quantities as
%
%
\begin{equation}\label{tt.23}
\sum_{\h\in B_\delta}\mu^h (\h)
{\mathbb P}_\h^h(\tau_A <\tau_{B_\delta})
\E_\h^h\tau_A .
\end{equation}
Applying the representation formula (\ref{intro.2}) for hitting times
for the $h$-trans\-formed dynamics,
we can represent
the above sum as
%
%
\begin{equation}\label{tt.4}
\sum_{\s' \in A_\delta\setminus A} \mu^h (\s' ){\mathbb P}_{\s
'}^h (\tau_{B_\delta}<\tau_A),
\end{equation}
and (\ref{eq:downhill}) follows.
\end{pf}

Let us go back to (\ref{eq:muestimate}).
Using an estimate completely analogous to Lem\-ma~\ref{lemma:hypothesis},
one sees that
%
%
\begin{eqnarray}\label{tt.444}\qquad
\sum_{\s' \in A_\delta\setminus A} \mu^h (\s' ){\mathbb P}_{\s
'}^h (
\tau_{B_\delta}<\tau_A)&\leq&
\sum_{\s' \in A_\delta\setminus A} \mu(\s' )h (\s' ){\mathbb
P}_{\s'} (
\tau_{B_\delta}<\tau_A)\nonumber\\[-8pt]\\[-8pt]
&\leq&\mu(A_\delta\setminus A)e^{-c_\delta{N}}
\nonumber
\end{eqnarray}
for some $c_\delta>0$.
This allows us to deduce (\ref{eq:muestimate}) from
(\ref{eq:downhill}).

\section*{Acknowledgments}
We thank Malwina Luczak for stimulating
discussions on~\cite{LLP08} and coupling methods in general.
We are grateful to Martin Slowik for pointing out an error in
a previous version and for suggesting the proof of Lemma
\ref{uphill.1}.

The kind hospitality of the Technion, Haifa, the Weierstrass-Institute
for Applied Analysis and Stochastics, and the Institute of Applied
Mathematics at Bonn University is gratefully acknowledged.

%

%
\printaddresses

\end{document}